\documentclass[12pt]{amsart}

\usepackage[a4paper,margin=1in]{geometry}
\usepackage{amsmath,amssymb,amsthm,mathtools}
\usepackage{enumitem}
\usepackage{hyperref}
\usepackage{tikz}
\usetikzlibrary{arrows.meta,calc,decorations.pathreplacing}
\usepackage{caption}
\usepackage{float}

\newtheorem{theorem}{Theorem}[section]
\newtheorem{proposition}[theorem]{Proposition}
\newtheorem{lemma}[theorem]{Lemma}
\newtheorem{corollary}[theorem]{Corollary}
\theoremstyle{remark}
\newtheorem{remark}[theorem]{Remark}

\newtheorem*{remark*}{Remark}
\newtheorem*{example*}{Example}
\newtheorem{definition}[theorem]{Definition}
\numberwithin{equation}{section}
\numberwithin{figure}{section}

\newcommand{\R}{\mathbb R}
\newcommand{\Sym}{\operatorname{Sym}}
\newcommand{\tr}{\operatorname{tr}}

\newcommand{\USC}{\operatorname{USC}}

\newcommand{\Om}{\Omega}

\title[Hessian Exterior Dirichlet Problems]{Existence and Nonexistence for Hessian Exterior Dirichlet Problems with \(k\)-Admissible Asymptotic Matrices}
\author[J.~G. Bao]{Jiguang Bao}
\author[Q.~F. Jiang]{Qinfeng Jiang$^*$}

\thanks{$^*$Corresponding author.\\
	Jiguang Bao is supported by the National Natural Science Foundation of China (12371200) and Beijing Natural Science Foundation (1254049).}

\begin{document}
\maketitle

\begin{center}
	\normalsize
	School of Mathematical Sciences, Beijing Normal University,\\
	Beijing 100875, China\\[0.3em]
	\texttt{jgbao@bnu.edu.cn}, \\
	\texttt{202531130031@mail.bnu.edu.cn}.
\end{center}
\begin{abstract}
	We study exterior Dirichlet problems for \(k\)-Hessian equations with prescribed
	quadratic asymptotics, allowing the asymptotic matrix to be merely
	\(k\)-admissible and not necessarily positive definite. The key point is that the correct metric at infinity is not determined by the asymptotic matrix itself, but by the coefficient matrix obtained by linearizing the \(k\)-Hessian operator at this matrix. 
	This gives the exterior barriers and subsolutions needed to solve the Dirichlet problem, both in viscosity and smooth settings, for all sufficiently large asymptotic constants.
	In the case of smooth, strictly star-shaped domains with strictly \((k-1)\)-convex boundary, we complete the characterization of existence and
	nonexistence through a linearized capacitary comparison and a
	tangential-trace contradiction on the inner boundary.
\end{abstract}

\medskip
\noindent\textbf{Keywords.}
\(k\)-Hessian equation; exterior Dirichlet problem; sharp asymptotic constant; \(k\)-admissible asymptotic matrix; strictly \((k-1)\)-convex boundary.

\medskip
\noindent\textbf{2020 Mathematics Subject Classification.}
Primary 35J60; Secondary 35J25, 35B40, 35B65.

\section{Introduction}\label{sec:introduction}

Let \(D\subset\R^n\) be a bounded domain, \(n\ge3\), and set
\(\Om:=\R^n\setminus\overline D\).  The object of study is the exterior Dirichlet
problem
\begin{equation}\label{eq:problem}
\begin{cases}
    \sigma_k(\lambda(D^2u))=1, & x\in \Om,\\
    u=\varphi, & x\in \partial D,
\end{cases}
\qquad 2\le k\le n-1,
\end{equation}
with prescribed quadratic asymptotics
\begin{equation}\label{eq:asymptotic}
    u(x)=\frac12x^TAx+b\cdot x+c+O(|x|^{2-n}),
    \qquad |x|\to\infty .
\end{equation}
 Here \(\lambda(D^2u)\) denotes the eigenvalues of the Hessian $D^2u$ and \(\sigma_k\)
is the \(k\)-th elementary symmetric function, 
defined by
\[
\sigma_k(\lambda)
=
\sum_{1\le i_1<\cdots<i_k\le n}
\lambda_{i_1}\cdots \lambda_{i_k},
\quad
\lambda=(\lambda_1,\ldots,\lambda_n)\in\R^n .
\]
 The elliptic branch is the
G\aa rding cone
\[
    \Gamma_k=\{\lambda\in\R^n:\ \sigma_j(\lambda)>0,\ j=1,\ldots,k\}.
\]
Denote the \(k\)-admissible matrix set by
\[
\mathcal M_k:=\{A\in\Sym(n):\lambda(A)\in\Gamma_k\}.
\]
In \eqref{eq:asymptotic}, the natural assumption on the asymptotic Hessian is
\begin{equation}\label{eq:A_assumption}
	A\in \mathcal A_k
	:=
	\{A\in\mathcal M_k:\sigma_k(\lambda(A))=1\}.
\end{equation}
and \(b\in \R^n\), \(c\in \R\). In the case \(k=n\), the \(k\)-Hessian equation coincides with the Monge--Amp\`ere equation \( \det D^2u = 1\).

Throughout the paper, a classical subsolution means a \(C^2\) function
\(v\) satisfying \(\lambda(D^2v)\in\Gamma_k\) and
\(\sigma_k(\lambda(D^2v))\ge1\). A strict subsolution means that the inequality is strict. A classical solution \(v\) is called \(k\)-admissible if \(D^2v\in\mathcal M_k\).
For viscosity solutions, \(k\)-admissibility means that the solution is taken on the \(\Gamma_k\) branch in the viscosity sense.

Global rigidity and asymptotic behavior form a natural background for exterior
Dirichlet problems with quadratic growth. For the Monge--Amp\`ere equation, the
J\"orgens--Calabi--Pogorelov theorem
\cite{Jorgens1954,Calabi1958,Pogorelov1972} asserts that every entire
classical convex solution of \(\det D^2u=1\) is a quadratic polynomial.  Caffarelli~\cite{Caffarelli1995} later extended the result for classical solutions to viscosity solutions. In the punctured setting, J\"orgens~\cite{Jorgens1955} first classified the planar smooth locally convex solutions, and Jin--Xiong~\cite{JinXiong2016} later extended this classification to generalized solutions in all dimensions. Their results show that, up to unimodular affine transformations, the solutions reduce to a one-parameter family of radial models.
  Caffarelli--Li \cite{CaffarelliLi2003} studied the corresponding exterior
asymptotic theory for convex viscosity solutions when the right-hand side is a
compact perturbation of \(1\); in dimension two, an additional logarithmic term
appears \cite{F1999,BXZ2019}. Further asymptotic results for perturbations
decaying to \(1\) at infinity were obtained by Bao--Li--Zhang
\cite{BaoLiZhang2015} and Liu--Bao \cite{LiuBao2022}, with later refinements
under weaker regularity assumptions \cite{QB2025,QB20252n} and a characterization of
remainders via a modified Kelvin transform \cite{HW2025}. For Hessian equations,
Liouville-type theorems and rigidity results have also been studied extensively:
Chang--Yuan \cite{ChangYuan2010} and Shankar--Yuan \cite{ShankarYuan2025}
proved rigidity for entire semiconvex solutions of the \(2\)-Hessian equation,
while Bao--Chen--Guan--Ji \cite{BaoChenGuanJi2003} took the lead in proving that strictly convex solutions of general \(k\)-Hessian equations with quadratic growth must be quadratic polynomials. This was subsequently refined under weaker convexity
assumptions; see, for example, Li--Ren--Wang \cite{LRW2016} and Wang--Bao
\cite{WB2022}.

The Dirichlet problem for fully nonlinear equations involving the eigenvalues
of the Hessian was initiated in bounded domains by
Caffarelli--Nirenberg--Spruck \cite{CNS1985} and Trudinger
\cite{Trudinger1995}; see also Guan \cite{Guan1994,Guan2023} for boundary
estimates and existence results based on admissible subsolutions. In exterior
domains, Caffarelli--Li \cite{CaffarelliLi2003} established existence and
uniqueness for the Monge--Amp\`ere equation with prescribed quadratic
asymptotics. This was extended to perturbative right-hand sides by
Bao--Li--Zhang \cite{BaoLiZhang2015}, and Li--Lu \cite{LiLu2018} later obtained
a sharp solvability criterion in terms of the additive asymptotic constant.
Recent works of Bao--Wang \cite{BaoWang2024,BW2024arxiv} further study
sharp boundary value conditions in the Monge--Amp\`ere case. For \(k\)-Hessian equations, the exterior Dirichlet problem was first studied by Dai--Bao \cite{DaiBao2011} and Dai \cite{Dai2011}. They considered the case where the asymptotic matrix is a multiple of the identity. Bao--Li--Li
\cite{BaoLiLi2014} then introduced generalized symmetric functions and treated the \(k\)-Hessian equation with a positive definite asymptotic matrix \(A>0\). After this, the case of positive definite asymptotic matrices \(A>0\) was studied in several exterior problems, including Hessian equations, Hessian quotient equations, special Lagrangian equations, and related equations; see Li--Bao \cite{LiBao2014}, Cao--Bao \cite{CaoBao2017}, Li--Li \cite{LiLi2018}, Li \cite{ZhisuLi2018}, Jiang--Li--Li \cite{JiangLiLi2021,JiangLiLi2022}, Li--Wang \cite{LiWang2024}, Dai--Bao--Wang \cite{DaiBaoWang2025}, and Bao--Jiang \cite{BJ2026}, among others.

The present paper continues this line of research by proving a sharp threshold
result for \(k\)-Hessian equations with asymptotic matrices that are
merely \(k\)-admissible, which is natural for intermediate
Hessian equations. For instance, when \(n=3\) and \(k=2\),
\[
A=\operatorname{diag}(1,1,0)
\quad\text{and}\quad
A=\operatorname{diag}(2,2,-3/4)
\]
both belong to \(\mathcal A_2\).  

To see the new difficulty, recall that a common feature of the existing exterior barrier constructions is the use of
generalized symmetric functions first introduced by Bao--Li--Li \cite{BaoLiLi2014}, depending on
\begin{equation}
	\label{ra}
	r_A(x)=(x^TAx)^{1/2},
\end{equation}
where $A\in \Sym(n)$ is the asymptotic matrix in \eqref{eq:asymptotic}.  This ansatz is natural when \(A>0\).  It becomes unsuitable when \(A\) is merely
\(k\)-admissible: if \(A\) has a null direction, \(r_A\) does not detect infinity
in that direction, while if \(A\) has a negative direction, \(r_A\) is not a real
exterior norm.  Thus an ansatz built directly from \(A\) treats the
asymptotic matrix and the metric at infinity as the same object. In the present
setting these two roles must be separated: \(A\) fixes the prescribed quadratic
asymptotics, while the metric at infinity and the exterior barriers are
governed by the positive definite coefficient matrix of the linearized operator at \(A\).

Let \(F(M):=\sigma_k(\lambda(M))\), \(M=(m_{ij})\in \mathcal{M}_k\). We denote by
\begin{equation}\label{Gnot}
	G=\bigl(G^{ij}\bigr)_{i,j=1}^n,
	\qquad
	G^{ij}:=
	\left.\frac{\partial F}{\partial m_{ij}}\right|_{M=A},
\end{equation}
the coefficient matrix of the linearized operator at \(A\). Since
\(A\in \mathcal{A}_k\), the matrix \(G\) is positive definite \cite{CNS1985}.  
Our main observation is that the metric at infinity is not determined by
\(r_A(x)\) in \eqref{ra}, which is built directly from the asymptotic matrix \(A\). Instead, it is determined by the coefficient matrix \(G\) of the linearized operator at \(A\). The corresponding exterior norm is
\begin{equation}\label{rho}
	\rho(x):=(x^TG^{-1}x)^{1/2}.
\end{equation}
In this setting, under the linear change of variables \(x=G^{1/2}y\), set
\(w(y)=v(G^{1/2}y)\). Then the linearized operator
\begin{equation}
	\label{LA}
L_Av:=G^{ij}v_{ij},
\end{equation}
satisfies \(L_Av(x)=\Delta_y w(y)\) and the norm \(\rho(x)\) becomes \(|y|\).
Thus the
expected fundamental solution correction has size \(\rho^{2-n}\).  This
viewpoint is consistent with asymptotic expansion results for existing exterior
solutions, such as Liu--Bao \cite{LiuBao2022,LBCPAA2022},  
Han--Marchenko \cite{HM2025} and Han--Wang \cite{HW2025}, but here it is used instead
to construct barriers.

The first ingredient is a far-field subsolution. Outside a sufficiently large ellipsoid, one proves that
\begin{equation}\label{farb}
	Q_c(x)-\Theta\rho(x)^{2-n}+\delta\rho(x)^{-n}
\end{equation}
is a classical subsolution for suitable
\(\Theta,\delta>0\), where \(Q_c(x)=x^TAx/2+b\cdot x+c\) is the prescribed quadratic asymptotic polynomial in \eqref{eq:asymptotic}. Here the norm \(\rho\) is defined by \eqref{rho}. The leading correction \(-\Theta\rho^{2-n}\) has the
natural order at infinity and is harmonic for the linearized operator \eqref{LA}. The lower-order term \(\delta\rho^{-n}\) is added to make the subsolution
inequality hold. Its contribution after linearization has the right sign and can be chosen to control the nonlinear error terms.

For the near-boundary subsolutions in the viscosity setting, we use a slope
form of quadratic supports. Bao--Li--Li \cite[Lemma 3.1]{BaoLiLi2014} write
their supports in the form
\begin{equation}
	\label{ctform}
	w_\xi(x)
	=
	\varphi(\xi)
	+\frac12\Big(
	(x-\bar x(\xi))^T A(x-\bar x(\xi))
	-
	(\xi-\bar x(\xi))^T A(\xi-\bar x(\xi))
	\Big),
	\; x\in\mathbb R^n,
\end{equation}
where \(\xi\in \partial D\). This form is convenient when \(A>0\). In the present setting,
\(A\) may be degenerate or indefinite, so we instead write the support as
\begin{equation}
	\label{sform}
	\omega_\xi(x)
	=
	\varphi(\xi)+p_\xi\cdot(x-\xi)
	+\frac12(x-\xi)^T M_\varepsilon(x-\xi),
	\qquad
	M_\varepsilon=A+\varepsilon G^{-1},
\end{equation}
where \(\varepsilon>0\) is chosen sufficiently small. The slope \(p_\xi\) is
chosen such that \(\omega_\xi\) touches \(\varphi\) from below at \(\xi\) and
stays below the boundary data elsewhere. This form is enough for the Perron
construction and does not require \(M_\varepsilon\) to be invertible.
Moreover, when \(M_\varepsilon\) is invertible, then by setting
\[
\bar x_\xi:=\xi-M_\varepsilon^{-1}p_\xi,
\]
one can rewrite \(\omega_\xi\) as \eqref{ctform}. Thus \eqref{sform} is the same type of quadratic support as \eqref{ctform},
but written in a way that does not require choosing the point \(\bar x_\xi\) in advance, and it remains useful even when the asymptotic matrix is not
invertible.

For continuous boundary data, we work with viscosity solutions and Perron's method. Following
the boundary-support viewpoint of Bao--Wang \cite{BaoWang2024,BW2024arxiv}, we use uniform semiconvexity of \(\varphi\) relative to \(\partial D\), together
with a uniform supporting-plane convexity condition on the boundary. The viscosity framework and
Perron construction are standard; the argument relies on the results of Ishii \cite{Ishii1989} and Crandall--Ishii--Lions \cite{CrandallIshiiLions1992}.

Combining the near-boundary supports with the far-field branch gives the first
main result.

\begin{theorem}[Viscosity exterior solution]\label{thm:viscosity_solution}
Let \(n\ge3\) and \(2\le k\le n-1\).  Suppose that \(D\subset\R^n\) is a bounded uniformly
convex \(C^{1,1}\) domain, and \(\varphi\) is uniformly
semiconvex with respect to \(\partial D\).   Then, for any given $A\in \mathcal{A}_k$ and \(b\in\R^n\), there exists
\(c_*=c_*(n,k,A,b,D,\varphi)>0\) such that for every \(c>c_*\), problem
\eqref{eq:problem} admits a unique $k$-admissible viscosity solution
\(u\in C^0(\overline\Om)\) satisfying \eqref{eq:asymptotic}.  
\end{theorem}

The existence theory for viscosity solutions typically relies on the construction of suitable subsolutions and supersolutions, and often requires certain convexity assumptions on the inner boundary.  Recently, Li--Xiao \cite{LiXiao2025} extended the smooth existence result of
Bao--Li--Li \cite{BaoLiLi2014} from strictly convex domains to star-shaped domains with strictly \((k-1)\)-convex boundary. The smooth theory is of a different nature.  We use the annular approximation and boundary estimates by Li--Xiao \cite{LiXiao2025} to handle the domains
with such properties.  The new point here is that the
smooth annular scheme remains compatible with a 
\(k\)-admissible asymptotic matrix \(A\).  The proof constructs a smooth global
subsolution from three pieces: a Li--Xiao local subsolution near \(\partial D\),
a linearized-norm bridge, and the far-field subsolution \eqref{farb}.  These pieces are joined by regularized
maxima and then used in a bounded-annulus approximation.

This smooth setting gives not only solvability for large asymptotic constants, but also a sharp threshold in that additive constant. The relevant boundary curvature notation is fixed in Section~\ref{subsec:boundary-capacity}.
Our second main result is as follows.

\begin{theorem}[Smooth exterior solutions and sharp asymptotic constants]\label{thm:smooth_sharp_threshold}
Let \(n\ge3\) and \(2\le k\le n-1\).  Suppose that \(D\subset\R^n\) is a smooth bounded, strictly star-shaped domain with
strictly \((k-1)\)-convex boundary.  Let \(\varphi\in C^\infty(\partial D)\).  Then, for any given \(A\in \mathcal{A}_k\) and 
\(b\in\R^n\), there exists a constant
\(
    C^*=C^*(n,k,A,b,D,\|\varphi\|_{C^2(\partial D)})\in\R
\)
such that the following hold.
\begin{enumerate}[label=(\roman*)]
    \item If \(c\ge C^*\), then \eqref{eq:problem} has a unique $k$-admissible
    solution \(u\in C^\infty(\overline{\Omega})\) satisfying \eqref{eq:asymptotic}.
    \item If \(c<C^*\), then \eqref{eq:problem} has no $k$-admissible solution
    \(u\in C^\infty(\overline{\Omega})\) satisfying 
    \eqref{eq:asymptotic}.
\end{enumerate}
\end{theorem}

Theorem~\ref{thm:smooth_sharp_threshold} shows that the additive constant \(c\) is the solvability parameter in the smooth exterior problem. Changing \(c\) only adds a constant to the asymptotic profile and leaves the equation unchanged. This leads to large-constant
solvability, and in the smooth setting it identifies the solvable asymptotic constants exactly as
the closed half-line \([C^*,\infty)\).  Moreover, there is a constant \(\underline c\le C^*\), depending only on
\(n,k,A,b,D,\|\varphi\|_{C^2(\partial D)}\), such that no \(k\)-admissible subsolution
\(u\in C^2(\overline{\Omega})\) satisfying
\[
u=\varphi\quad\text{on }\partial D,\qquad
u(x)=Q_c(x)+o(1),\quad |x|\to\infty,
\]
exists when \(c<\underline c\).

The same construction also gives a more explicit asymptotic estimate than
the prescribed behavior in \eqref{eq:asymptotic} for the large  solutions. More precisely, for all sufficiently large \(c\), the solutions obtained in Theorems~\ref{thm:viscosity_solution} and \ref{thm:smooth_sharp_threshold} satisfy
\begin{equation}
	\label{msexp}
	Q_c(x)-Cc^{n-1}|x|^{2-n}\le u(x)\le Q_c(x), \quad |x|\ge c,
\end{equation}
where \(C\) is a constant depending only on the fixed data, but not on \(c\). To the best of our
knowledge, such an explicit estimate in terms of the additive asymptotic
constant has not appeared before, even for the Monge--Amp\`ere equation.

\begin{remark}
	Strict \((k-1)\)-convexity of the boundary alone does not imply that the
	domain is star-shaped. Indeed, for every \(n\ge3\), there are smooth
	contractible domains in \(\mathbb R^n\) whose boundaries are strictly
	\((n-2)\)-convex but which are not star-shaped; see Appendix
	\ref{app:c-shaped-example}. Since strict \((n-2)\)-convexity implies strict
	\((k-1)\)-convexity for every \(2\le k\le n-1\), the star-shapedness
	assumption on \(D\) in Theorem~\ref{thm:smooth_sharp_threshold} is independent
	of the curvature assumption.
\end{remark}

%

The nonexistence part of Theorem~\ref{thm:smooth_sharp_threshold} is based on a
capacitary argument in the linearized metric. If \(u\) is a subsolution and
\(w=u-Q_c\), then the concavity of \(\sigma_k^{1/k}\) gives
\(L_Aw\ge0\). Comparing \(w\) with the
exterior \(L_A\)-capacitary potential gives a negative exterior normal derivative
at a boundary maximum of \(w\). Taking the tangential trace of the boundary
Hessian identity then forces \(\tr_{T\partial D}D^2u<0\) for very negative
\(c\). This contradicts a simple algebraic consequence of
\(k\)-admissibility: for \(k\ge2\), every \(k\)-admissible matrix has
nonnegative trace.

\begin{remark}\label{rem:boundary_hypotheses}
	The preceding argument is a smooth boundary argument. It uses the boundary
	normal derivative at a contact point with the \(L_A\)-capacitary potential, and
	this is not directly available for merely continuous viscosity solutions with
	the ordinary Dirichlet condition. The curvature input is also weaker than uniform convexity:
	positive definiteness of the second fundamental form is not required. The proof
	uses only the trace condition \eqref{eq:H0_intro},
	which follows from strict \((k-1)\)-convexity.
\end{remark}

The large-constant branch obtained by annular approximation also has the
standard higher-order decay.

\begin{theorem}\label{thm:higher_asymptotics}
Let \(u\) be the smooth solution obtained for sufficiently large \(c\)  under the hypotheses of
Theorem~\ref{thm:smooth_sharp_threshold}.  Set
\[
    E(x):=u(x)-\left(\frac12x^TAx+b\cdot x+c\right).
\]
Then, for every integer \(m\ge1\),
\begin{equation}\label{eq:higher_asymptotics}
    \limsup_{|x|\to\infty}|x|^{n-2+m}|D^mE(x)|<\infty .
\end{equation}
\end{theorem}


Finally, we point out that the role of the linearized matrix \eqref{Gnot} is not
merely technical; it reflects a general principle. This viewpoint is not specific to the Hessian operator.  The same linearized
metric should govern exterior problems for other elliptic spectral equations,
such as Hessian quotient equations and special Lagrangian equations, whenever
the linearized matrix at the prescribed asymptotic Hessian is positive definite.


The paper is organized as follows.  Section~\ref{sec:preliminaries} collects the
comparison principle, the linearized metric, far-field barriers, a parameter
selection lemma, the regularized maximum, and the boundary and capacitary facts
used in the sharp-threshold
argument.  Section~\ref{sec:viscosity} proves Theorem~\ref{thm:viscosity_solution}
by Perron's method.  In Sections~\ref{sec:smooth_subsolutions} and
\ref{sec:annular_large}, we construct smooth strict global subsolutions, solve the
annular problems, pass to the exterior limit, prove the higher-order asymptotics in Theorem \ref{thm:higher_asymptotics},
and record the ordered exterior solvability statement needed later.  Finally,
Section~\ref{sec:sharp_constant} gives the capacitary nonexistence estimate and
uses it, together with interpolation and compactness, to establish
Theorem~\ref{thm:smooth_sharp_threshold}.

\section{Preliminaries}\label{sec:preliminaries}

\subsection{The linearized metric, comparison, and far-field barriers}

Throughout, write
\[
    F(M)=\sigma_k(\lambda(M)),\qquad F(A)=1.
\]
Since the equation depends only on the Hessian, the linear term in the
asymptotic polynomial can be absorbed into the boundary data. Without loss of generality, we assume \(b=0\) in \eqref{eq:asymptotic}, hence
\[
    Q_c(x):=\frac12x^TAx+c.
\]

The following comparison principle is used repeatedly. Its proof is standard, and is included only for completeness.
\begin{lemma}[Comparison principle]\label{lem:comparison}
Let \(U\subset\R^n\) be a bounded domain.  Let
\(u,v\in C^2(U)\cap C^0(\overline U)\) be \(k\)-admissible, and suppose
\[
    \sigma_k(\lambda(D^2u))\ge1,
    \qquad
    \sigma_k(\lambda(D^2v))\le1
    \quad\text{in }U.
\]
If \(u\le v\) on \(\partial U\), then \(u\le v\) in \(U\).  The same conclusion
holds on exterior domains provided
\(\limsup_{|x|\to\infty}(u-v)\le0\).
\end{lemma}

Recall the notions in \eqref{Gnot} and \eqref{rho}.
The positivity of \(G\) follows from \(\lambda(A)\in\Gamma_k\).  For \(R>0\), set
\[
    E_R:=\{x\in\R^n:\ \rho(x)<R\}.
\]
Since \(G>0\), each \(E_R\) is a smooth strictly convex ellipsoid.  Under the
linear change of variables \(y=G^{-1/2}x\), the norm \(\rho\) becomes \(|y|\),
and the linearized operator defined in \eqref{LA}
becomes the Euclidean Laplacian in the \(y\)-variables.

Both one-sided and two-sided versions of the far-field barriers will be used below.
For \(\Theta>0\) and \(\delta>0\), define
\begin{equation}
	\label{uinf}
	u_{\infty}(x):=Q_c(x)-\Theta\rho(x)^{2-n}+\delta\rho(x)^{-n}.
\end{equation}

\begin{lemma}[Far-field barriers]\label{lem:far_field}
Assume $A\in \mathcal{A}_k$.  There exists a structural constant
\(R_{*}>0\), depending only on \(n,k,A,\Theta\) and \(\delta\),  such that for $R>R_{*}$,
 \(u_{\infty}\) is a smooth \(k\)-admissible subsolution
in \(\R^n \setminus E_{R}\), and
\begin{equation}
	\label{asneq}
	 Q_c(x)-C|x|^{2-n}<u_{\infty}(x)<Q_c(x),
	\qquad \text{in}\; \R^n \setminus \overline{E}_{R},
\end{equation}
for a constant \(C=C(A,n,k,\Theta)>0\).

In addition, 
\[
    Q_c+\Theta\rho^{2-n}-\delta\rho^{-n}
\]
is a smooth \(k\)-admissible supersolution in \(\R^n \setminus E_{R}\).
\end{lemma}

\begin{proof}
 For every real \(\alpha\), differentiating \(\rho^\alpha=(x^TG^{-1}x)^{\alpha/2}\) gives
\[
    \partial_{ij}\rho^\alpha
    =\alpha\rho^{\alpha-2}(G^{-1})_{ij}
    +\alpha(\alpha-2)\rho^{\alpha-4}(G^{-1}x)_i(G^{-1}x)_j .
\]
Since \(|G^{-1}x|\le C\rho\), it follows that
\(
    |\partial_{ij}\rho^\alpha|\le C\rho^{\alpha-2},
\)
where \(C\) depends only on \(n\), \(\alpha\) and the eigenvalue bounds 
\(A\).  In particular,
With the change of variables
\(y=G^{-1/2}x\), one has \(\rho(x)=|y|\) and \(L_A=\Delta_y\).  Hence
\[
L_A(\rho^{2-n})=0,
\qquad
L_A(\rho^{-n})=2n\rho^{-n-2}.
\]

Set
\[
\psi(x):=-\Theta\rho(x)^{2-n}+\delta\rho(x)^{-n}.
\]
Then \(D^2u_{\infty}=A+D^2\psi\) and 
\begin{equation}\label{eq:LA_psi_minus}
	L_A\psi=2n\delta\rho^{-n-2}.
\end{equation}
On the other hand,
\begin{equation}\label{eq:D2psi_bound}
    |D^2\psi|\le C_1\Theta\rho^{-n}+C_2\delta\rho^{-n-2}.
\end{equation}
Since \(F\) is smooth near \(A\), there exist \(\eta>0\) and \(C_A>0\) such that
whenever \(H\in \Sym(n)\) and \(|H|\le\eta\),
\[
    F(A+H)=F(A)+L_A H+O(|H|^2).
\]
The admissibility follows from the openness of \(\Gamma_k\).  Since
\(\lambda(A)\in\Gamma_k\), there exists \(\eta_0>0\) such that \(|H|<\eta_0\)
implies \(\lambda(A+H)\in\Gamma_k\).  Choosing \(R\ge R_{*}\) large enough and using
\eqref{eq:D2psi_bound}, we obtain \(\lambda(A+D^2\psi)\in\Gamma_k\) in
\(\{\rho\ge R\}\).  Thus \(u_{\infty}\) is a classical
\(k\)-admissible subsolution.

Taking \(H=D^2\psi\), for large $R_{*}$ so that \(|D^2\psi|\le\eta\) for
\(\rho\ge R\), we obtain from \eqref{eq:LA_psi_minus} and \eqref{eq:D2psi_bound}, 
\[
    F(A+D^2\psi)=1+L_A(D^2\psi)+O(|D^2\psi|^2)
    \ge 1+2n\delta\rho^{-n-2}-O(\rho^{-2n}).
\]
Hence, for \(\rho\ge R\)
\[
    F(A+D^2\psi)\ge1.
\]

Finally, since \(\rho\) and \(|x|\) are comparable, for \(R\) sufficiently large,
\[
    -C|x|^{2-n}< -\Theta\rho^{2-n}+\delta\rho^{-n}<0.
\]
This gives the stated two-sided estimate for the subsolution branch.

The supersolution  is proved in the same way, with
\[
\psi_+(x):=\Theta\rho^{2-n}-\delta\rho^{-n}.
\]
The linear term changes sign, and the same Taylor expansion shows that, for
\(\rho\) sufficiently large, the negative linear contribution dominates the
quadratic remainder; hence \(F(A+D^2\psi_+)\le 1\), with admissibility preserved
after increasing \(R\) if necessary.
\end{proof}

%

The preceding barrier lemma will be used through the following large-constant
parameter choice.  Recording it here keeps the later gluing arguments short.

Fix \(K_0>0\). For parameters \(0<\tau<1\), \(\Lambda>0\), and \(c>0\), set
\begin{equation}\label{RTC}
	R_c:=\tau c, \qquad \delta_c:=\Lambda c^2 R_c^{\,n-2}.
\end{equation}
 If \(P_{1,c},P_{2,c}\in\mathbb R\) satisfy 
 \begin{equation}\label{P12}
 	|P_{1,c}|+|P_{2,c}|\le K_0(1+R_c),
 \end{equation}
  define
  \begin{equation}
  	\label{Thepm}
  	\Theta_c^-:=(c+\delta_c R_c^{-n}-P_{1,c})R_c^{\,n-2}, \qquad \Theta_c^+:=(c+\delta_c(2R_c)^{-n}-P_{2,c})(2R_c)^{\,n-2},
  \end{equation}
and put \( \mathcal I_c:=[\Theta_c^-,\Theta_c^+]\).

\begin{lemma}[Parameter selection]\label{lem:farfield_parameter_selection}  
There exists \(\tau_0\in(0,1)\) such that, for every \(0<\tau\le\tau_0\), one can choose \(\Lambda>0\) and \(c_0>0\) with the following property. For all \(c\ge c_0\), the interval \(\mathcal I_c\) is nonempty, and every \(\Theta_c\in\mathcal I_c\) satisfies \[ \frac12 cR_c^{\,n-2}\le \Theta_c\le CcR_c^{\,n-2}, \] where \(C\) is independent of \(c\). Moreover, \[  u_\infty^c:=Q_c-\Theta_c\rho^{2-n}+\delta_c\rho^{-n} \] is a smooth $k$-admissible subsolution in  \(\R^n \setminus E_{R_c}\), and
\begin{equation}
	\label{seq}
	u_\infty^c \le Q_c-\frac12\Theta_c\rho^{2-n}<Q_c, \qquad \text{in}\;  \R^n \setminus E_{R_c} .
\end{equation}
\end{lemma}

\begin{proof}
By \eqref{Thepm},  a direct subtraction gives
\[
    \Theta_c^+-\Theta_c^-
    =R_c^{n-2}\Big[(2^{n-2}-1)c+P_{1,c}-2^{n-2}P_{2,c} 
    +(2^{-2}-1)\delta_cR_c^{-n}\Big].
\]
Using \eqref{P12}, and the identity
\(\delta_cR_c^{-n}=\Lambda\tau^{-2}\) from \eqref{RTC}, we obtain
\[
    \Theta_c^+-\Theta_c^-
    \ge R_c^{n-2}\big[(2^{n-2}-1)c-C_1(1+\tau c)-\Lambda\tau^{-2}\big].
\]
Choose \(\tau_0>0\) so small that the coefficient of \(c\) remains positive
for every \(0<\tau\le\tau_0\).  Once \(\tau\) and \(\Lambda\) are fixed, choose
\(c_0\) large enough.  This proves that the interval is nonempty.
The same estimates give
\[
    \Theta_c^-\ge \frac12 cR_c^{n-2},
    \qquad
    \Theta_c^+\le CcR_c^{n-2},
\]
after increasing \(c_0\) if necessary.

It remains to check compatibility with Lemma~\ref{lem:far_field}.  From
\(\Theta_c\sim c^{n-1}\) and the definition of \(\delta_c\),
\[
    \Theta_cR_c^{-n}\le \frac{C}{\tau^2c},
    \qquad
    \delta_cR_c^{-n-2}=\frac{\Lambda}{\tau^4c^2},
\]
which are small for \(c\) large.  Also
\[
    \frac{\Theta_c^2R_c^{-2n}}{\delta_cR_c^{-n-2}}\le \frac{C}{\Lambda}.
\]
Thus choosing \(\Lambda\) large, Lemma~\ref{lem:far_field} implies $ u_\infty^c$ is a smooth $k$-admissible subsolution.  Finally, using \(\Theta_c\geq \frac12 cR_c^{n-2}\), 
\[
    \frac{\delta_c\rho^{-n}}{\Theta_c\rho^{2-n}}
    =\frac{\delta_c}{\Theta_c\rho^{2}}
    \le \frac{\delta_c}{\Theta_cR_c^2}
    \le \frac{C\Lambda}{\tau^2c}\leq \frac{1}{2},
\]
holds for \(\rho> R_c\).  Hence \eqref{seq} follows after increasing \(c_0\) once more.
\end{proof}

\subsection{Regularized maximum}

Smooth subsolutions will be glued by a smooth maximum, using the following standard device.
\begin{lemma}[Regularized maximum]\label{remax}
	Let \(\mu>0\).  There exists a function \(M_\mu\in C^\infty(\R^2)\) with
	\begin{enumerate}[label=(\roman*)]
		\item \(M_\mu\) convex and nondecreasing in each variable;
		\item \(M_\mu(s,t)=\max\{s,t\}\) whenever \(|s-t|\ge\mu\);
		\item \(\max\{s,t\}\le M_\mu(s,t)\le\max\{s,t\}+\mu\).
	\end{enumerate}
\end{lemma}

\begin{proof}
	Choose an even function \(\eta\in C_c^\infty((-1,1))\), \(\eta\ge0\), with
	\[
	\int_{\mathbb R}\eta(r)\,dr=1,
	\]
	and set
	\[
	\eta_\mu(r):=\mu^{-1}\eta(r/\mu),\qquad
	\chi_\mu(r):=\int_{\mathbb R}|r-s|\eta_\mu(s)\,ds .
	\]
	Then \(\chi_\mu\in C^\infty(\mathbb R)\), \(\chi_\mu\) is even and convex, and
	\[
	\chi_\mu''=2\eta_\mu\ge0,\qquad |\chi_\mu'|\le1 .
	\]
	Since \(\eta_\mu\) is even and supported in \((-\mu,\mu)\), we also have
	\[
	\chi_\mu(r)=|r| \quad\text{for } |r|\ge\mu .
	\]
	Moreover, Jensen's inequality gives \(\chi_\mu(r)\ge |r|\), while the triangle
	inequality gives
	\[
	\chi_\mu(r)\le |r|+\int_{\mathbb R}|s|\eta_\mu(s)\,ds\le |r|+\mu .
	\]
	
	Define
	\[
	M_\mu(s,t):=\frac{s+t+\chi_\mu(s-t)}2 .
	\]
	Since \(\chi_\mu\) is convex, \(M_\mu\) is convex. Also,
	\[
	\partial_s M_\mu=\frac{1+\chi_\mu'(s-t)}2,\qquad
	\partial_t M_\mu=\frac{1-\chi_\mu'(s-t)}2,
	\]
	and hence \(M_\mu\) is nondecreasing in each variable. If \(|s-t|\ge\mu\), then
	\(\chi_\mu(s-t)=|s-t|\), hence
	\[
	M_\mu(s,t)=\frac{s+t+|s-t|}{2}=\max\{s,t\}.
	\]
	Finally,
	\[
	0\le M_\mu(s,t)-\max\{s,t\}
	=
	\frac{\chi_\mu(s-t)-|s-t|}{2}
	\le \frac{\mu}{2}\le\mu .
	\]
	This proves the stated properties.
\end{proof}

Using the above Lemma, we then get
\begin{lemma}\label{lem:regmax}  Let $M_{\mu}$ as in Lemma \ref{remax}.
Suppose \(u_1,u_2\in C^2(\R^n)\) satisfy
\[
    \lambda(D^2u_i)\in\Gamma_k,
    \qquad
    \sigma_k(\lambda(D^2u_i))\ge1,
    \qquad i=1,2,
\]
then \(M_\mu(u_1,u_2)\in C^2(\R^n)\), and
\[
    \lambda(D^2M_\mu(u_1,u_2))\in\Gamma_k,
    \qquad
    \sigma_k(\lambda(D^2M_\mu(u_1,u_2)))\ge1\quad \text{in}\; \R^n.
\]
If, on a set \(K\Subset\Omega\), the stronger inequalities
\(\sigma_k(\lambda(D^2u_i))\ge1+\delta\) hold for both \(i=1,2\), then
\(\sigma_k(\lambda(D^2M_\mu(u_1,u_2)))\ge1+\delta\) on \(K\).
\end{lemma}

\begin{proof}
By Lemma \ref{remax}, it's clear that $M_\mu(u_1,u_2)\in C^2(\R^n)$.     Let \(r=u_1-u_2\) and put
\[
    \alpha:=\frac{1+\chi_\mu'(r)}2,
    \qquad
    \gamma:=\frac12\chi_\mu''(r)\ge0.
\]
Then \(0\le\alpha\le1\) and direct differentiation gives
\begin{equation}\label{eq:regmax_hessian}
    D^2M_\mu(u_1,u_2)=\alpha D^2u_1+(1-\alpha)D^2u_2
    +\gamma D(u_1-u_2)\otimes D(u_1-u_2).
\end{equation}
The cone \(\Gamma_k\) is convex, and adding a nonnegative semidefinite matrix
preserves the \(\Gamma_k\)-branch.  Hence \(\lambda(D^2M_\mu(u_1,u_2))\in\Gamma_k\).  Let
\(
    \widehat F(M):=\sigma_k(\lambda(M))^{1/k}.
\)
By the Caffarelli--Nirenberg--Spruck concavity theorem \cite{CNS1985},
\(\widehat F(M)\) is concave for \(M\in \mathcal{M}_{k}\) and is nondecreasing in nonnegative
symmetric directions.  Therefore, from \eqref{eq:regmax_hessian},
\[
    \widehat F(D^2M_\mu(u_1,u_2))
    \ge \widehat F(\alpha D^2u_1+(1-\alpha)D^2u_2)
    \ge \alpha \widehat F(D^2u_1)+(1-\alpha)\widehat F(D^2u_2).
\]
Since \(\sigma_k(\lambda(D^2u_i))\ge1\), the right-hand side is at least \(1\),
and hence
\(\sigma_k(\lambda(D^2M_\mu(u_1,u_2)))\ge1\). If both input branches satisfy
\(\sigma_k(\lambda(D^2u_i))\ge1+\delta\) on \(K\), then
\(\widehat F(D^2u_i)\ge(1+\delta)^{1/k}\). The same inequality gives
\(\widehat F(D^2M_\mu(u_1,u_2))\ge(1+\delta)^{1/k}\), and hence
\(\sigma_k(\lambda(D^2M_\mu(u_1,u_2)))\ge1+\delta\) on \(K\). This proves the lemma.
\end{proof}

\subsection{Boundary and capacitary preliminaries for the sharp threshold}\label{subsec:boundary-capacity}

The following boundary convention is used throughout the smooth part of the paper. Let \(\nu\) be the exterior unit normal to \(\partial D\), pointing from \(D\) into \(\Omega\). For \(\xi\in\partial D\), let
\[
T_\xi\partial D:=\{v\in\mathbb R^n: v\cdot\nu(\xi)=0\}
\]
be the tangent space of \(\partial D\) at \(\xi\). Our convention for the second fundamental form is
\[
II(\tau,\eta):=\langle D_\tau\nu,\eta\rangle,
\qquad \tau,\eta\in T_\xi\partial D.
\]
With this convention, the principal curvatures of a sphere with respect to the exterior normal are positive. We denote by \(\kappa(\xi)=(\kappa_1(\xi),\ldots,\kappa_{n-1}(\xi))\)
the eigenvalues of \(II\) at \(\xi\), namely the principal curvatures of
\(\partial D\) at \(\xi\). The boundary \(\partial D\) is called strictly
\((k-1)\)-convex if
\[
\kappa(\xi)\in\Gamma_{k-1}^{(n-1)}
\qquad\text{for every }\xi\in\partial D,
\]
where \(\Gamma_{k-1}^{(n-1)}\) is the G\r{a}rding cone in \(\mathbb R^{n-1}\).
In particular,
\begin{equation}\label{eq:H0_intro}
	\mathcal H_0:=\min_{\xi \in \partial D}\sigma_1(\kappa(\xi))>0.
\end{equation}
For \(k=2\), this is exactly strict mean convexity; for \(k>2\), it is a
stronger curvature condition which still implies \(\mathcal H_0>0\).

The trace notation and elementary boundary identities needed for the
nonexistence proof are recorded next. For a symmetric matrix \(M\), and for any
orthonormal basis \(\tau_1,\ldots,\tau_{n-1}\) of \(T_\xi\partial D\), write
\begin{equation}\label{eq:tangential_trace}
	\operatorname{tr}_{T_\xi\partial D}M
	:=
	\sum_{\alpha=1}^{n-1}\tau_\alpha^T M\tau_\alpha
	=
	\operatorname{tr}M-\nu(\xi)^TM\nu(\xi).
\end{equation}
This quantity is independent of the choice of the orthonormal basis of
\(T_\xi\partial D\). Define
\begin{equation}
	\label{eq:A_boundary_plus}
	A_\partial^+
	:=
	\max\left\{0,\sup_{\xi\in\partial D}
	\operatorname{tr}_{T_\xi\partial D}A\right\}.
\end{equation}

Let \(w\in C^2\) in a neighborhood of \(\partial D\). We denote by
\(D^2_{\partial D}w\) the intrinsic Hessian of the restriction
\(w|_{\partial D}\), and write \(\partial_\nu w:=\nabla w\cdot\nu\).
With the above convention for \(II\), the relation between the ambient Hessian and the boundary Hessian is
\begin{equation}\label{eq:ambient_boundary_hessian}
	D^2w(\tau,\tau)
	=
	D^2_{\partial D}w(\tau,\tau)
	+
	\partial_\nu w\,II(\tau,\tau),
	\qquad \tau\in T_\xi\partial D .
\end{equation}

 The smooth sharp-threshold argument also uses the following capacitary
potential in the linearized metric.  Let \(h\) be the exterior capacitary potential for \(L_A\):
\begin{equation}\label{eq:capacity_problem}
\begin{cases}
    L_Ah=0, & x\in\Omega,\\
    h=1, & x\in\partial D,\\
    h(x)\to0, & |x|\to\infty.
\end{cases}
\end{equation}
This problem has a unique solution, smooth in \(\Omega\) and smooth up to
\(\partial D\), satisfying \(0<h<1\) in \(\Omega\).  Indeed, the linear change of
variables \(y=G^{-1/2}x\) transforms \(L_A\) into the Euclidean Laplacian, so
\eqref{eq:capacity_problem} is equivalent to the classical exterior Dirichlet
problem for the Laplace equation outside the smooth bounded set \(G^{-1/2}D\)
\cite{GT1983}.  The solution is the equilibrium, or capacitary, potential of
this compact set.  Equivalently, it is obtained as the limit of the solutions in
bounded annuli with boundary values \(1\) on the inner boundary and \(0\) on the
outer boundary.  The maximum principle gives \(0<h<1\), and comparison with a
multiple of \(\rho^{2-n}\) gives
\[
    h(x)=O(|x|^{2-n}),
    \qquad |x|\to\infty.
\]
By the Hopf lemma,
\begin{equation}\label{eq:gamma}
    \gamma:=\min_{\partial D}(-\partial_\nu h)>0.
\end{equation}

\medskip\noindent Two elementary facts will be used in the sharp nonexistence proof.

\begin{lemma}\label{lem:subharmonic}
Let \(u\in C^2(\Omega)\) be a \(k\)-admissible subsolution of the Hessian
equation and suppose
\[
    u(x)=Q_c(x)+o(1),
    \qquad |x|\to\infty.
\]
Set \(w:=u-Q_c\).  Then
\[
    L_Aw\ge0\quad\text{in }\Omega.
\]
\end{lemma}

\begin{proof}
Let \(\widehat F(M):=\sigma_k(\lambda(M))^{1/k}\) as before.  Since \(u\) is a
subsolution,
\[
    \widehat F(D^2u)\ge1=\widehat F(A).
\]
By the concavity of \(\widehat F(M)\) for \(M\in \mathcal{M}_{k}\),
\[
    \widehat F(D^2u)-\widehat F(A)  
    \le \operatorname{tr}(D\widehat F(A) \cdot \left(D^2u-A\right)) .
\]
Hence
\[
    \operatorname{tr}(D\widehat F(A) \cdot \left(D^2u-A\right))\ge0.
\]
Since
\[
    D\widehat F(A)=k^{-1}\sigma_k(\lambda(A))^{1/k-1}DF(A)=k^{-1}DF(A),
\]
this is exactly \(L_Aw\ge0\).
\end{proof}

The second fact is purely algebraic.  It will be applied at the boundary point where the capacitary comparison touches.

\begin{lemma}\label{lem:hyperplane_trace}
Let \(\nu\in\R^n\) be a unit vector.  If \(M\in \mathcal{M}_{2}\), then
\[
    \tr M-\nu^TM\nu>0.
\]
If \(M\in \overline{\mathcal{M}}_{2}\), then the same quantity is nonnegative.
\end{lemma}

\begin{proof}
This follows from the standard positivity property of the G\aa rding cone
\(\Gamma_2\); see Caffarelli--Nirenberg--Spruck \cite{CNS1985}.
\end{proof}

\section{Viscosity exterior solutions}\label{sec:viscosity}

The construction of the global viscosity subsolution for Perron's method starts
from the near-boundary piece, which is a maximum of quadratic lower supports
whose Hessian is a small perturbation of the prescribed asymptotic matrix in the
linearized metric.  For \(\varepsilon>0\), set
\[
    M_\varepsilon:=A+\varepsilon G^{-1}.
\]
Since \(\lambda(A)\in\Gamma_k\) and \(\Gamma_k\) is open, after fixing \(\varepsilon_0=\varepsilon_0(n,k,A)>0\) sufficiently small, every
\(0<\varepsilon\le\varepsilon_0\) satisfies \(\lambda(M_\varepsilon)\in\Gamma_k\).  Moreover,
\begin{align*}
    F(M_\varepsilon)
    &=F(A+\varepsilon G^{-1})\\
    &=F(A)+\varepsilon G^{ij}(G^{-1})_{ij}+O(\varepsilon^2)\\
    &=1+n\varepsilon+O(\varepsilon^2)>1 .
\end{align*}

The two elementary boundary and data assumptions used in the barrier construction are as follows.

\begin{definition}[Uniform convexity in the supporting-plane sense]
	\label{def:uniform_boundary_convexity}
	Let \(D\subset\mathbb R^n\) be a bounded domain.  The domain
	\(D\) is said to be uniformly convex in the supporting-plane sense if there exist
	constants \(r_D,c_{D},\kappa_D,\xi_D>0\) such that, for every
	\(\xi,x\in\partial D\),
	\begin{align}
		-c_{D}|x-\xi|^2\leq(x-\xi)\cdot\nu(\xi)
		&\le -\kappa_D |x-\xi|^2,
		&& |x-\xi|<r_D, \label{eq:uniform_convex_local}\\
		(x-\xi)\cdot\nu(\xi)
		&\le -\xi_D,
		&& |x-\xi|\ge r_D, \label{eq:uniform_convex_far}
	\end{align}
	where \(\nu(\xi)\) denotes the exterior unit normal to \(\partial D\) at
	\(\xi\).
\end{definition}

\begin{remark}\label{rem:uniform_convex_supporting_plane}
		The supporting-plane condition above is satisfied by balls, ellipsoids, and more generally, it is
		also satisfied by \(C^{1,1}\) uniformly convex domains in the quantitative
		sense that, in local coordinates centered at each boundary point \(\xi\), with
		the exterior normal as the vertical direction, \(\partial D\) is trapped between
		two downward quadratic graphs with uniform constants. 
	  Thus the uniform convexity assumption in
Theorem~\ref{thm:viscosity_solution} supplies the constants used below.
\end{remark}

\begin{definition}[Uniform semiconvexity relative to the boundary]
	\label{def:boundary_semiconvexity}
	Let \(D\subset\R^n\) be a bounded \(C^{1,1}\) domain. A function
	\(\varphi\in C^0(\partial D)\) is called uniformly semiconvex with respect to
	\(\partial D\) if there exist constants \(r_\varphi, K_\varphi, P_\varphi>0\)
	such that, for every \(\xi\in\partial D\), there exists a vector
	\[
	\tau_\xi\in T_\xi\partial D,\qquad |\tau_\xi|\le P_\varphi,
	\]
	depending on \(\varphi\), such that
	\begin{equation}\label{eq:boundary_semiconvex_support}
		\varphi(x)
		\ge
		\varphi(\xi)+\tau_\xi\cdot(x-\xi)-K_\varphi |x-\xi|^2
	\end{equation}
	whenever \(x\in\partial D\) and \(|x-\xi|<r_\varphi\).
\end{definition}

\begin{remark}
		The vector \(\tau_\xi\) should be viewed as a tangential lower supporting slope
		of the boundary data \(\varphi\) at \(\xi\). In a local \(C^{1,1}\) boundary chart, it
		corresponds to a subgradient of the local representative of the semiconvex
		function \(\varphi\), equivalently after adding a fixed quadratic function. Thus \(\tau_\xi\) depends on \(\varphi\), while the
		condition \(\tau_\xi\in T_\xi\partial D\) records that only tangential
		first-order variations are intrinsic for a function defined on \(\partial D\).
		If \(\varphi\in C^1(\partial D)\), one may take
		\(\tau_\xi=\nabla_{\partial D}\varphi(\xi)\).
\end{remark}

These boundary hypotheses are used only to produce uniform quadratic supports from below.

\begin{lemma}[Quadratic boundary supports]\label{lem:vis_boundary_supports} Let $D$ and $\varphi$
under the hypotheses of Theorem~\ref{thm:viscosity_solution}, for every
\(0<\varepsilon\le\varepsilon_0\) and every \(\xi\in\partial D\) there is a
vector \(p_\xi^\varepsilon\), uniformly bounded, depending only on $n, k, A, D, K_{\varphi}$ and $P_{\varphi}$,
such that
\[
    \omega_\xi^\varepsilon(x)
    :=\varphi(\xi)+p_\xi^\varepsilon\cdot(x-\xi)
    +\frac12(x-\xi)^TM_\varepsilon(x-\xi),\quad x\in \R^n,
\]
satisfies
\[
    \omega_\xi^\varepsilon(\xi)=\varphi(\xi),
    \qquad
    \omega_\xi^\varepsilon<\varphi
    \quad\text{on }\partial D\setminus\{\xi\}.
\]
Moreover, each \(\omega_\xi^\varepsilon\) is a strict smooth \(k\)-admissible subsolution.
\end{lemma}

\begin{proof}
The number \(\varepsilon_0>0\) has been fixed so that
\(\|M_\varepsilon\|\le c(n,k,A)\)  for \(0<\varepsilon\le\varepsilon_0\).
Choose
\[
    p_\xi^\varepsilon:=\tau_\xi+N\nu(\xi),
\]
where \(N>0\) will be chosen later.  If
\(0<|x-\xi|<r_0:=\min\{r_D,r_\varphi\}\), then 
\eqref{eq:uniform_convex_local} and \eqref{eq:boundary_semiconvex_support} give
\begin{align*}
    \omega_\xi^\varepsilon(x)-\varphi(x)
    &\le N\nu(\xi)\cdot(x-\xi)
       +K_\varphi|x-\xi|^2
       +\frac12\|M_\varepsilon\||x-\xi|^2  \\
    &\le -\Bigl(N\kappa_D-K_\varphi-\frac12\|M_\varepsilon\|\Bigr)|x-\xi|^2 .
\end{align*}
Taking \(N=N(n,k,A,\kappa_D,K_\varphi)\) sufficiently large makes
this negative for every \(0<|x-\xi|<r_0\).

It remains to consider the compact part
\(
    \{x\in\partial D:|x-\xi|\ge r_0\}.
\)
Since \(\varphi\) is continuous,  \(\partial D\) is
compact,  there exists a constant
\(C_{\partial}\), independent of \(\xi\) and \(\varepsilon\), such that
\[
    \tau_\xi\cdot(x-\xi)
    +\frac12(x-\xi)^TM_\varepsilon(x-\xi)
    -\bigl(\varphi(x)-\varphi(\xi)\bigr)
    \le C_{\partial}
\]
on this set.  The global separation in Definition~\ref{def:uniform_boundary_convexity}
then yields
\[
    \omega_\xi^\varepsilon(x)-\varphi(x)
    \le -N\xi_D+C_{\partial}.
\]
Increasing \(N\) once more, still only in terms of the fixed data, makes this strictly negative.  Hence
\(\omega_\xi^\varepsilon\) touches \(\varphi\) from below at \(\xi\) and is
strictly below \(\varphi\) at every other boundary point.  Finally,
\(D^2\omega_\xi^\varepsilon=M_\varepsilon\), and the strict smooth \(k\)-admissible subsolution property
follows from \(M_\varepsilon\in\mathcal{M}_k\) and \(F(M_\varepsilon)>1\).
\end{proof}

Define
\[
    w_\varepsilon(x):=\sup_{\xi\in\partial D}\omega_\xi^\varepsilon(x)\qquad x\in  \R^n.
\]
Since the family \(\{\omega_\xi^\varepsilon\}_{\xi\in\partial D}\) is locally
uniformly Lipschitz in \(x\), with constants independent of \(\xi\), the function
\(w_\varepsilon\) is continuous. By the stability of viscosity subsolutions under locally bounded suprema \cite{CrandallIshiiLions1992},
\(w_\varepsilon\) is a viscosity subsolution.  The touching property of the family
\(\{\omega_\xi^\varepsilon\}\) gives
\[
    w_\varepsilon=\varphi
    \quad\text{on }\partial D.
\]

  The functions \(w_\varepsilon\) are close to the target quadratic profile on
large linearized ellipsoids.  More precisely, there is a constant \(C_0>0\),
independent of \(\varepsilon\), such that
\begin{equation}\label{eq:vis_w_estimate}
	|w_\varepsilon(x)-Q_0(x)|
	\le C_0(1+\rho(x))+\frac{\varepsilon}{2}\rho(x)^2, \quad x\in \R^n,
\end{equation}
where \(Q_0(x)=Q_c(x)-c\). Indeed, expanding \(\omega_\xi^\varepsilon-Q_0\) gives
\begin{align*}
	\omega_\xi^\varepsilon(x)-Q_0(x)
	&=\frac12x^T(M_\varepsilon-A)x
	+(p_\xi^\varepsilon-M_\varepsilon\xi-b)\cdot x  \\
	&\quad
	+\varphi(\xi)-p_\xi^\varepsilon\cdot\xi
	+\frac12\xi^TM_\varepsilon\xi .
\end{align*}
Since \(\partial D\) is compact and the vectors
\(p_\xi^\varepsilon\) are uniformly bounded for
\(0<\varepsilon\le\varepsilon_0\), we have
\[
|p_\xi^\varepsilon-M_\varepsilon\xi-b|\le C,
\qquad
\left|
\varphi(\xi)-p_\xi^\varepsilon\cdot\xi
+\frac12\xi^TM_\varepsilon\xi
\right|\le C, \quad \xi\in\partial D,
\]
with \(C\) independent of \(\xi\) and \(\varepsilon\).  Therefore, by \(|x|\le C(A)\rho(x)\),
\[
\left|
(p_\xi^\varepsilon-M_\varepsilon\xi-b)\cdot x
+\varphi(\xi)-p_\xi^\varepsilon\cdot\xi
+\frac12\xi^TM_\varepsilon\xi
\right|
\le C(1+|x|)\leq C\left(1+\rho(x)\right).
\]
Since 
\[
\frac12x^T(M_\varepsilon-A)x
=\frac{\varepsilon}{2}x^TG^{-1}x
=\frac{\varepsilon}{2}\rho(x)^2,
\]
we obtain, uniformly in \(\xi\in \partial D\),
\[
\omega_\xi^\varepsilon(x)-Q_0(x)
\le C(1+\rho(x))+\frac{\varepsilon}{2}\rho(x)^2 .
\]
Taking the supremum over \(\xi\in\partial D\) gives the upper bound for
\(w_\varepsilon-Q_0\).

For the lower bound, fix some \(\xi_0\in\partial D\).  Since
\(w_\varepsilon\ge \omega_{\xi_0}^\varepsilon\), the same expansion gives
\[
w_\varepsilon(x)-Q_0(x)
\ge
\omega_{\xi_0}^\varepsilon(x)-Q_0(x)
\ge
-C(1+\rho(x))+\frac{\varepsilon}{2}\rho(x)^2 .
\]
Combining the two inequalities proves \eqref{eq:vis_w_estimate}.

The boundary part is then pasted with the far-field branch.

\begin{proposition}[Viscosity global subsolution]\label{prop:vis_global_subsolution}
Under the hypotheses of Theorem~\ref{thm:viscosity_solution}, for every
\(b\in\R^n\) there exists \(c_*>0\) such that, for each \(c>c_*\), there is a \(k\)-admissible viscosity subsolution \(\underline u_c\in C(\overline{\Om})\) with
\[
    \underline u_c=\varphi\quad\text{on }\partial D,
    \qquad
    \underline u_c\le Q_c,
\]
and
\[
    \underline u_c(x)=Q_c(x)+O(|x|^{2-n})
    \quad\text{as } |x|\to\infty .
\]
Moreover, \(\underline u_c\) satisfies the estimate
\[
Q_c(x)-C c^{\,n-1}|x|^{2-n}
\le \underline u_c(x)\le Q_c(x),
\qquad |x|\ge c,
\]
where \(C\) is a constant depending only on the fixed data, but not on \(c\).
\end{proposition}

\begin{proof}
Let \(m\) and \(M\) denote the corresponding minimum on \(\partial E_R\)
and maximum on \(\partial E_{2R}\) in \eqref{eq:vis_w_estimate}. Set \(\varepsilon=R^{-n}\). Then there exists a constant \(K_0>0\) such that  
\[
    |m|+|M|\le K_0(1+R).
\]
In view of 
Lemma~\ref{lem:farfield_parameter_selection}, by choosing \(\tau>0\) sufficiently
small, then choose \(\Lambda\) and \(c_0\).  For \(c\ge c_0\), set
\[
    R_c:=\tau c,
    \qquad
    \varepsilon_c:=R_c^{-n},
    \qquad
    \delta_c:=\Lambda c^2R_c^{n-2}.
\]
Write \(w_c:=w_{\varepsilon_c}\), and set
\[
    m_c:=\min_{\partial E_{R_c}}(w_c-Q_0),
    \qquad
    M_c:=\max_{\partial E_{2R_c}}(w_c-Q_0).
\]
Then \((P_{1,c},P_{2,c})=(m_c,M_c)\) satisfies \eqref{P12}. Therefore, by Lemma \ref{lem:farfield_parameter_selection}, there exists
\(\Theta_c\in \mathcal I_c\), where
\[
    \Theta_c^-=(c+\delta_cR_c^{-n}-m_c)R_c^{n-2},
    \qquad
    \Theta_c^+=(c+\delta_c(2R_c)^{-n}-M_c)(2R_c)^{n-2}.
\]
Set
\[
    u_\infty^c:=Q_c-\Theta_c\rho^{2-n}+\delta_c\rho^{-n}.
\]
Lemma~\ref{lem:farfield_parameter_selection} gives that \(u_\infty^c\) is a
smooth \(k\)-admissible subsolution, satisfies
\(u_\infty^c<Q_c\) in \(\R^n \setminus E_{R_c}\), and has the prescribed quadratic asymptotics.  The choice of
\(\Theta_c\) gives the interface inequalities
\begin{equation}
	\label{intereq}
	 u_\infty^c\le w_c\quad\text{on }\partial E_{R_c},
	\qquad
	u_\infty^c\ge w_c\quad\text{on }\partial E_{2R_c}.
\end{equation}

Define
\[
\underline u_c(x)=
\begin{cases}
    w_c(x), & x\in E_{R_c}\setminus\overline D,\\[2mm]
    \max\{w_c(x),u_\infty^c(x)\}, & x\in E_{2R_c}\setminus E_{R_c},\\[2mm]
    u_\infty^c(x), & x\in \R^n\setminus E_{2R_c}.
\end{cases}
\]
The interface inequalities \eqref{intereq} make $\underline u_c\in C(\overline{\Om})$.  In the middle
annulus the function is the maximum of two viscosity subsolutions, and the
maximum stability of viscosity subsolutions \cite{CrandallIshiiLions1992} shows
that \(\underline u_c\) is a viscosity subsolution. Since \(w_c\) and \(u_\infty^c\) are smooth \(k\)-admissible subsolutions, their maximum is a \(k\)-admissible viscosity subsolution on the \(\Gamma_k\) branch, and hence \(\underline u_c\) is a \(k\)-admissible subsolution in the  viscosity sense.  The boundary value comes
from \(w_c=\varphi\) on \(\partial D\), and the asymptotic expansion comes from \eqref{asneq} and Lemma \ref{lem:farfield_parameter_selection}.

It remains to ensure that \(w_c\le Q_c\) in the bounded part.  For
\(x\in E_{2R_c}\setminus\overline D\), estimate \eqref{eq:vis_w_estimate} gives
\[
    w_c(x)-Q_c(x)
    \le C(1+\rho(x))+\frac{\varepsilon_c}{2}\rho(x)^2-c.
\]
Since \(\rho(x)<2R_c\), \(\varepsilon_c=R_c^{-n}\), and \(R_c=\tau c\), we have
\[
    w_c(x)-Q_c(x)\le C_0+C_1\tau c-c.
\]
Since Lemma~\ref{lem:farfield_parameter_selection} allows \(\tau\) to be chosen
arbitrarily small below \(\tau_0\), we choose it at the start so that
\(C_1\tau\le1/2\).  After increasing \(c_0\) such that \(C_0\le c/2\), the last
quantity is nonpositive.  Hence \(w_c\le Q_c\) in
\(E_{2R_c}\setminus\overline D\).  Combined with Lemma~\ref{lem:farfield_parameter_selection},
therefore \(\underline u_c\le Q_c\) in \(\overline{\Om}\).

Recall that \(G>0\), and let
\[
\alpha_G:=\lambda_{\min}(G^{-1})^{1/2}
=\lambda_{\max}(G)^{-1/2},
\]
then
\(
\rho(x)\ge \alpha_G |x|
\).
We choose \(\tau_0>0\) in Lemma \ref{lem:farfield_parameter_selection} so small that
\(
\tau_0\le \alpha_G/2
\).
Then, for every \(0<\tau\le\tau_0\) and \(R_c=\tau c\),
\[
|x|\geq c
\quad\Longrightarrow\quad
\rho(x)\ge 2\tau |x|\geq 2R_c .
\]
Hence,  for $|x|\ge c$, \(\underline{u}_c=u_\infty^c\). By Lemmas \ref{lem:far_field} and \ref{lem:farfield_parameter_selection}, we have 
\[
Q_c(x)\geq \underline{u}_c(x)\geq Q_c(x)-Cc^{n-1}|x|^{2-n}, \quad |x|\geq c,
\]
and
\[
\underline{u}_c(x)
=Q_c(x)+O(|x|^{2-n})
\quad\text{as } |x|\to\infty.
\]
This finishes the proof.
\end{proof}

 \begin{remark}
	If one used a fixed positive definite matrix, such as $G^{-1}$, for the boundary barrier, then the near-boundary quadratic function would grow like $x^TG^{-1}x$ at infinity.  If $A$ has a zero or negative direction, this growth is incompatible with the prescribed asymptotic quadratic polynomial $x^TAx/2+b\cdot x+c$.  The mismatch on a large interface would be of order $R_c^2$.  By using instead
	\[
	M_{\varepsilon_c}=A+\varepsilon_c G^{-1},
	\qquad
	\varepsilon_c\sim R_c^{-n},
	\]
	the quadratic mismatch is reduced to $\varepsilon_c R_c^2=O(R_c^{2-n})$, while the remaining linear errors are of order $O(R_c)$ and can be controlled by the large constant $c$ with $R_c=\tau c$.
\end{remark}

This ordered subsolution--supersolution pair allows Perron's method to prove Theorem~\ref{thm:viscosity_solution}.

\begin{proof}[Proof of Theorem~\ref{thm:viscosity_solution}]
The quadratic polynomial
\[
    \overline u(x):=Q_c(x)=\frac12x^TAx+b\cdot x+c
\]
satisfies
\[
    D^2\overline u=A,
    \qquad
    \lambda(A)\in\Gamma_k,
    \qquad
    F(D^2\overline u)=F(A)=1.
\]
Thus \(\overline u\) is a classical solution and, in particular, a viscosity
supersolution.  Increasing \(c\), if necessary, we may also assume
\(\overline u\ge\varphi\) on \(\partial D\).  Proposition
\ref{prop:vis_global_subsolution} gives a \(k\)-admissible viscosity subsolution \(\underline u_c\) such
that
\[
    \underline u_c=\varphi\quad\text{on }\partial D,
    \qquad
    \underline u_c\le Q_c\quad\text{in }\Omega,
    \qquad
    \underline u_c(x)=Q_c(x)+O(|x|^{2-n}).
\]
Therefore \(\underline u_c\) and \(\overline u\) form an ordered subsolution--supersolution
pair.  Define the Perron class
\[
\begin{aligned}
    \mathcal S_c:=\{v\in\USC(\overline\Omega):&\ v\text{ is a \(k\)-admissible viscosity subsolution of }\sigma_k(\lambda(D^2v))=1,\\
    &\ \underline u_c\le v\le\overline u\text{ in }\Omega,
      \quad v=\varphi\text{ on }\partial D\}.
\end{aligned}
\]
This class is nonempty because \(\underline u_c\in\mathcal S_c\).  Set
\[
    u(x):=\sup_{v\in\mathcal S_c}v(x).
\]
By the standard Perron method and the comparison principle on the
\(\Gamma_k\) branch, the function \(u\) is continuous, \(k\)-admissible  and solves
\eqref{eq:problem} in the viscosity sense, with the prescribed quadratic
asymptotics \eqref{eq:asymptotic}. Uniqueness follows from the same comparison principle: one compares two solutions on truncated exterior domains and then lets the outer norm tend to infinity. This is the standard viscosity Perron framework used in Bao--Li--Li \cite{BaoLiLi2014} and Li--Wang \cite{LiWang2024}. 
\end{proof}

\section{Smooth strict global subsolutions}\label{sec:smooth_subsolutions}

This section works under the hypotheses of
Theorem~\ref{thm:smooth_sharp_threshold}.  Since \(D\) is strictly
star-shaped with respect to the origin, its radial function
\(\rho_D\in C^\infty(S^{n-1})\) is positive.  Define the radial gauge of
\(D\) by
\begin{equation*}
    \beta(x):=\frac{|x|}{\rho_D(x/|x|)},
    \qquad x\in\R^n\setminus\{0\}.
\end{equation*}
Then $\beta=1$ on $\partial D$, $\beta>1$ in $\Om$, and $\beta$ is smooth in $\R^n\setminus \{0\}$.

The following lemma is a local, bounded-region version of the Li--Xiao \cite{LiXiao2025}
near-boundary construction.  It provides a smooth strict \(k\)-admissible
subsolution in $U\setminus \overline{D}$,  where \(U\) is a fixed bounded
neighborhood of \(\overline D\).

\begin{lemma}[Li--Xiao type local subsolution]\label{lem:LX_local}
Let $U$ be a smooth bounded domain in \(\R^n\) such that \(\overline D\subset U\).
Let $\Phi\in C^\infty(\overline U\setminus D)$ be a smooth extension of $\varphi$, namely
\begin{equation*}
    \Phi=\varphi\quad\text{on }\partial D.
\end{equation*}
Then there exists a constant $N\gg1$ such that
\begin{equation*}
    v_0(x):=\Phi(x)+\beta(x)^N-1
\end{equation*}
satisfies
\begin{equation*}
    v_0=\varphi\quad\text{on }\partial D,
\end{equation*}
and
\begin{equation*}
    \lambda(D^2v_0)\in\Gamma_k,
    \qquad
    \sigma_k(\lambda(D^2v_0))>1
    \quad\text{in }\overline U\setminus\overline D.
\end{equation*}
\end{lemma}

We now introduce the bridge which connects the Li--Xiao local subsolution to the far-field barrier.

\begin{lemma}[Linearized-norm bridge]\label{lem:bridge_subsolution}
Let $a>0$ and $C_b\in\R$.  Then
\begin{equation*}
    v_{\mathrm{br}}(x):=Q_0(x)+a\rho(x)+C_b \in C^{\infty}(\R^n\setminus\{0\}),
\end{equation*}
and 
\begin{equation*}
    \lambda(D^2v_{\mathrm{br}})\in\Gamma_k,
    \qquad
    \sigma_k(\lambda(D^2v_{\mathrm{br}}))\ge1.
\end{equation*}
In particular, the inequality is strict wherever \(D^2\rho\ne0\).
\end{lemma}

\begin{proof}
  Set \(y=G^{-1/2}x\). Then \(\rho(x)=|y|\).  A direct computation gives
\begin{equation*}
    D_x^2\rho
    =G^{-1/2}\left(\frac{I}{|y|}-\frac{y\otimes y}{|y|^3}\right)G^{-1/2}
    \ge0  \quad \text{at}\; y\neq 0,
\end{equation*}
in the sense of symmetric matrices.  

 Using \eqref{eq:A_assumption}, we get
\begin{equation*}
    \lambda(D^2v_\mathrm{br})\in\Gamma_k,
    \qquad
    \sigma_k(\lambda(D^2v_\mathrm{br}))\ge\sigma_k(\lambda(A))=1.
\end{equation*}
Thus $v_{\mathrm{br}}$ is a classical subsolution.
\end{proof}

The Li--Xiao local subsolution can now be glued to the bridge.  Choose fixed radii \(0<r_1<r_2\) such that
\[
\overline D\subset E_{r_1},
\qquad
E_{r_2}\subset U.
\]
Here \(U\) is the fixed bounded domain used in Lemma
\ref{lem:LX_local}.  

  Since \(v_0-Q_0\) is continuous on the compact hypersurfaces
  \(\partial E_{r_1}\) and \(\partial E_{r_2}\), set
  \[
  m_1:=\min_{\partial E_{r_1}}(v_0-Q_0),
  \qquad
  M_2:=\max_{\partial E_{r_2}}(v_0-Q_0).
  \]
  Fix \(\mu>0\) small, and choose \(a>0\) so large that
  \[
  a(r_2-r_1)>M_2-m_1+6\mu .
  \]
  The quantities \(m_1,M_2,r_1,r_2\) are already fixed, so \(a\) is fixed
  independently of the large parameter \(c\). Define
  \[
  C_b:=m_1-2\mu-ar_1 .
  \]
  Then, from Lemma \ref{lem:bridge_subsolution}, we have on \(\partial E_{r_1}\),
  \[
  v_{\rm br}-Q_0=m_1-2\mu
  \le v_0-Q_0-2\mu,
  \]
  and on \(\partial E_{r_2}\),
  \[
  v_{\rm br}-Q_0
  =m_1-2\mu+a(r_2-r_1)
  >M_2+4\mu
  \ge v_0-Q_0+4\mu .
  \]

The inequalities above are strict with a positive margin.  Therefore, by continuity, 
the same inequalities hold in neighborhoods of the two interfaces within
\(E_{r_2}\setminus E_{r_1}\). That is, there exist neighborhoods
\(\mathcal N_1\) of \(\partial E_{r_1}\) and \(\mathcal N_2\) of
\(\partial E_{r_2}\), both contained in \(E_{r_2}\setminus E_{r_1}\),  such that
\[
v_0\ge v_{\mathrm{br}}+\mu
\quad\text{in }\mathcal N_1,
\]
and
\[
v_{\mathrm{br}}\ge v_0+\mu
\quad\text{in }\mathcal N_2.
\]
This persistence of the strict inequalities is the point that makes the smooth
gluing possible.  Indeed, the regularized maximum \(M_\mu\) in Lemma \ref{lem:regmax} gives
\[
M_\mu(v_0,v_{\mathrm{br}})=v_0
\quad\text{in }\mathcal N_1,
\]
and
\[
M_\mu(v_0,v_{\mathrm{br}})=v_{\mathrm{br}}
\quad\text{in }\mathcal N_2.
\]

This leads to the first glued smooth function
\begin{equation}
	\label{v1}
	v_1(x):=
	\begin{cases}
		v_0(x), & x\in E_{r_1}\setminus\overline D,\\[2mm]
		M_\mu(v_0(x),v_{\mathrm{br}}(x)), & x\in E_{r_2}\setminus E_{r_1},\\[2mm]
		v_{\mathrm{br}}(x), & x\in \mathbb R^n\setminus E_{r_2},
	\end{cases}
\end{equation}
at this stage before the far-field subsolution is inserted.

This piecewise definition gives a smooth function.  Near \(\partial E_{r_1}\),
the middle expression equals \(v_0\), and hence matches the inner definition
smoothly.  Near \(\partial E_{r_2}\), it equals \(v_{\mathrm{br}}\), and hence
matches the outer definition smoothly.  Thus no corner or loss of
differentiability is created at either interface.

Moreover, by Lemma~\ref{lem:regmax}, the function \(v_1\) is a smooth
\(k\)-admissible subsolution and satisfies
\[
\lambda(D^2v_1)\in\Gamma_k,
\qquad
\sigma_k(\lambda(D^2v_1))\ge1
\]
wherever \(v_1\) is defined.  Finally, since \(v_1=v_0\) in a neighborhood of
\(\partial D\), and \(v_0=\varphi\) on \(\partial D\), we have
\[
v_1=\varphi
\quad\text{on }\partial D.
\]

The next step is to choose the far-field branch and paste it to the bridge.  For
\(\Theta>0\) and \(\delta>0\) we use \eqref{uinf}
as a subsolution in the far-field region by Lemma~\ref{lem:far_field}
provided the parameters satisfy the smallness and dominance conditions.

It remains to choose the parameters needed to connect the bridge to the
far-field subsolution.  Since \(\mu\), \(a\) and \(C_b\) are fixed, the interface
heights below are bounded by \(K_0(1+R_c)\) for a fixed \(K_0\).  Applying
Lemma~\ref{lem:farfield_parameter_selection}, we choose \(0<\tau<1\) so small
that also \(2a\tau\le1/2\), choose \(\Lambda>0\), and then take \(c\)
sufficiently large and set $R_c$  and $\delta_c$ as in \eqref{RTC}.

Define
\[
    P_{1,c}:=aR_c+C_b-2\mu,
    \qquad
    P_{2,c}:=2aR_c+C_b+2\mu.
\]
Let \(\Theta_c^-\) and \(\Theta_c^+\) be the corresponding notions in
\eqref{Thepm}.  By Lemma~\ref{lem:farfield_parameter_selection}, choose
\(\Theta_c\in \mathcal{I}_{c}\) and set
\[
u_\infty^c(x):=Q_c(x)-\Theta_c\rho^{2-n}+\delta_c\rho^{-n}.
\]
One can verify
\begin{equation}
	\label{eq:outer_interface_requirements}
	u_\infty^c\le v_{\mathrm{br}}-2\mu
	\quad\text{on }\partial E_{R_c},
	\qquad
	u_\infty^c\ge v_{\mathrm{br}}+2\mu
	\quad\text{on }\partial E_{2R_c}.
\end{equation}
  Moreover,
Lemma~\ref{lem:farfield_parameter_selection} gives
 that \(u_\infty^c\) is a smooth
\(k\)-admissible subsolution in \(\R^n\setminus E_{R_c}\).  It also gives
\begin{equation}
    \label{eq:u_infty_below_Q}
    u_\infty^c\le Q_c
    \quad\text{for }\rho\ge R_c,
    \qquad
    u_\infty^c(x)=Q_c(x)+O(|x|^{2-n}).
\end{equation}

The interface inequalities \eqref{eq:outer_interface_requirements} allow us to
glue the bridge to the far-field branch by Lemma \ref{lem:regmax}.  Define
\[
v_2:=M_\mu(v_{\mathrm{br}},u_\infty^c)
\quad\text{in }E_{2R_c}\setminus E_{R_c},
\]
and set \(v_2=v_{\mathrm{br}}\) for \(\rho\le R_c\) and
\(v_2=u_\infty^c\) for \(\rho\ge2R_c\).  By
Lemma~\ref{lem:regmax}, \(v_2\) is a smooth strict \(k\)-admissible subsolution.

After increasing \(c\), if necessary, we may assume that \(R_c>r_2\).  Combining the inner and outer gluing steps, define
\[
\underline{u}_c(x)=
\begin{cases}
    v_1(x), & x\in E_{R_c}\setminus\overline D,\\
    v_2(x), & x\in E_{2R_c}\setminus E_{R_c},\\
    u_\infty^c(x), & x\in \R^n\setminus E_{2R_c}.
\end{cases}
\]
The strict separation inequalities at the interfaces make this function smooth.
By Lemmas~\ref{lem:LX_local}, \ref{lem:bridge_subsolution},
\ref{lem:far_field}, and \ref{lem:regmax},
\[
    \lambda(D^2\underline{u}_c)\in\Gamma_k,
    \qquad
    \sigma_k(\lambda(D^2\underline{u}_c))\ge1
    \quad\text{in }\; \Om.
\]
Moreover, \(\underline{u}_c=\varphi\) on \(\partial D\), because
\(\underline{u}_c=v_0\) near \(\partial D\) and \(v_0=\varphi\) there.

 Similar to the proof in Proposition \ref{prop:vis_global_subsolution}, we also have 
\[
Q_c(x)\geq \underline{u}_c(x)\geq Q_c(x)-Cc^{n-1}|x|^{2-n}, \quad |x|\geq c,
\]
 and
\[
\underline{u}_c(x)
=Q_c(x)+O(|x|^{2-n})
\quad\text{as } |x|\to\infty.
\]

It remains only to check that \(\underline{u}_c\le Q_c\).  The Li--Xiao local subsolution
and the first gluing region are contained in the fixed bounded set
\(E_{r_2}\setminus\overline D\).  Hence there is a constant \(C_0\), independent of the
large parameter \(c\), such that \(v_0-Q_0\le C_0\) there.  Taking
\(c>C_0+2\mu\) gives
\[
    v_0\le Q_c-2\mu
    \quad\text{in }E_{r_2}\setminus\overline D.
\]
For the bridge,
\[
    v_{\mathrm{br}}-Q_c=a\rho+C_b-c.
\]
Since \(\rho\le2R_c=2\tau c\) in \(E_{r_2}\setminus\overline D\), the
choice \(2a\tau\le1/2\) and a further increase of \(c\) give
\[
    v_{\mathrm{br}}\le Q_c-2\mu
    \quad\text{for }\rho\le2R_c.
\]
The far-field inequality is \eqref{eq:u_infty_below_Q}.  Finally, the
regularized maximum increases the ordinary maximum by at most \(\mu\), while in
each gluing region the two relevant branches have been arranged to lie below
\(Q_c-2\mu\).  Hence
\[
    \underline{u}_c\le Q_c
    \quad\text{in }\; \overline{\Om}.
\]

\begin{remark}\label{rem:bridge_role}
The role of the bridge function \(v_{\mathrm{br}}\) is to connect two incompatible scales.
The Li--Xiao subsolution is strong near \(\partial D\) but has the wrong large
scale if extended too far, while the far-field subsolution has the correct prescribed
asymptotics but is too high near the inner boundary when \(c\) is large.  \(v_{\mathrm{br}}\) has only linear growth in the linearized norm, is a subsolution since
\(D^2\rho\ge0\), and can be made to cross both neighboring pieces at prescribed
interfaces.
\end{remark}

The preceding construction proves the following result.

\begin{proposition}[Smooth global subsolution]\label{prop:smooth_global_subsolution}
Let $n\ge3$ and $2\le k\le n-1$.  Let $D\subset\R^n$ be a smooth bounded domain which is strictly star-shaped with respect to the origin, with $0\in D$, and whose boundary is strictly $(k-1)$-convex.  Let $\varphi\in C^\infty(\partial D)$.
For any given $A\in \mathcal{A}_k$ and $b\in\R^n$, there exists $c^*>0$ such that for every $c>c^*$ there is a smooth function
\begin{equation*}
    \underline{u}_c\in C^\infty(\overline{\Omega})
\end{equation*}
satisfying
\begin{enumerate}[label=\textup{(\roman*)}]
    \item $\lambda(D^2\underline{u}_c)\in\Gamma_k$ and $\sigma_k(\lambda(D^2\underline{u}_c))\ge1$ in $\Omega$;
    \item $\underline{u}_c=\varphi$ on $\partial D$;
    \item $\underline{u}_c\le Q_c$ in $\Omega$;
    \item $ Q_c(x)-Cc^{n-1}|x|^{2-n}\le \underline{u}_c(x)\le Q_c(x)$,  $|x|\ge c$,
\end{enumerate}
where \(C\) is a constant depending only on the fixed data, but not on \(c\). Moreover, $\underline{u}_c$ agrees with the far-field expression
\begin{equation*}
    Q_c(x)-\Theta_c\rho(x)^{2-n}+\delta_c\rho(x)^{-n}
\end{equation*}
for all sufficiently large $\rho(x)$.
\end{proposition}

\begin{remark}[Strictness on bounded annuli]\label{rem:strictness_annuli}
For constructing the global exterior subsolution, the non-strict inequality
\[
    \sigma_k(\lambda(D^2\underline{u}_c))\ge1
\]
is enough.  For the classical bounded annular problem, however, we use the fact
that the same construction is a strict subsolution on each fixed annulus.  The Li--Xiao local
subsolution is strict by Lemma \ref{lem:LX_local}.  The bridge is strict in the annular region
where it is used, since
\( D^2\rho\ge0,
\)
and \(D^2\rho\not\equiv0\). Monotonicity of \(\sigma_k\) in nonnegative
directions then gives a strict inequality at each point of the bridge region.
The far-field subsolution is strict on bounded subannuli of \(\R^n\setminus E_{R_c}\) by the
positive first-order term \(2n\delta_c\rho^{-n-2}\) in
\eqref{eq:LA_psi_minus}.  On any fixed annulus these strict inequalities have
a positive minimum.  The strict form of Lemma~\ref{lem:regmax} shows that the
regularized maximum preserves this positive margin when both input branches are
strict.  Hence, for every \(S>2R_c\), there exists \(\delta_S>0\) such that
\[
    \sigma_k(\lambda(D^2\underline{u}_c))\ge1+\delta_S
    \quad\text{on }\overline{E}_{S}\setminus D.
\]
\end{remark}

\begin{remark}\label{rem:smooth_vs_viscosity}
This smooth construction is stronger than an upper-semicontinuous viscosity
pasting.  It is therefore suitable for the a priori estimates on bounded
annuli, in the spirit of the estimates used after Li--Xiao's construction of a
strict subsolution.  The constants in the bridge depend on the fixed inner
annulus and on the chosen gluing margin; the large parameter \(c\) is used only
after these choices have been made.
\end{remark}
\section{Annular approximation and asymptotics}\label{sec:annular_large}
The smooth global subsolution from Proposition \ref{prop:smooth_global_subsolution} is now used in a classical annular approximation scheme.
At this point all parameters in the global construction, in particular $c$ and $R_c=\tau c$, have already been fixed. 

 Let $S>2R_c$ be arbitrary and set
\(
    \Omega_S:=E_S\setminus\overline D.
\)
Since $G^{-1}>0$, each $E_S$ is a smooth strictly convex ellipsoid.  On $\Omega_S$ we consider the bounded Dirichlet problem
\begin{equation}\label{eq:annulus_problem}
\begin{cases}
    \sigma_k(\lambda(D^2u_S))=1, & x\in\Omega_S,\\
    u_S=\varphi, & x\in\partial D,\\
    u_S=u_\infty^c, & x\in\partial E_S.
\end{cases}
\end{equation}
  Since $S>2R_c$, the global subsolution $\underline{u}_c$ constructed above agrees with $u_\infty^c$ on $\partial E_S$.  Hence
\begin{equation*}
    \underline{u}_c=\varphi\quad\text{on }\partial D,
    \qquad
    \underline{u}_c=u_\infty^c\quad\text{on }\partial E_S.
\end{equation*}
Moreover, by Proposition \ref{prop:smooth_global_subsolution} (ii) and (iii),
\begin{equation}\label{eq:supersolution_boundary_annulus}
    Q_c\ge \varphi\quad\text{on }\partial D,
    \qquad
    Q_c\ge u_\infty^c\quad\text{on }\partial E_S.
\end{equation}
Since $D^2Q_c=A$ and $F(A)=1$, the function $Q_c$ is a classical solution of the same equation.  Hence $(\underline{u}_c,Q_c)$ gives an ordered subsolution--supersolution pair for \eqref{eq:annulus_problem}.

\begin{proposition}[Solvability on bounded annuli]\label{prop:bounded_annulus_solution}
For every sufficiently large $S>2R_c$, problem \eqref{eq:annulus_problem} admits a unique strictly smooth  $k$-admissible solution.
\end{proposition}
\begin{proof}
	Fix \(S>2R_c\).  The annulus
	\[
	\Omega_S=E_S\setminus\overline D
	\]
	is smooth and bounded.  The outer boundary \(\partial E_S\) is a strictly
	convex ellipsoid in the linearized metric, while the inner boundary is
	\(\partial D\), which is strictly \((k-1)\)-convex.  By the
	CNS--Trudinger--Guan \cite{CNS1985,Trudinger1995,Guan1994} continuity method for the \(k\)-Hessian Dirichlet problem, in the non-convex annular form used by
	Li--Xiao \cite{LiXiao2025}, it suffices to exhibit a smooth strict \(k\)-admissible
	subsolution agreeing with the boundary data.
	The bounded-domain existence theorem therefore yields a smooth \(k\)-admissible
	solution \(u_S\) of \eqref{eq:annulus_problem}.
	Uniqueness follows from the comparison principle.
\end{proof}

The estimates needed for passage to the exterior limit are recorded next.  Constants near the inner boundary and on compact subsets of $\Omega$ are independent of $S$.

The first estimate is the ordering obtained from comparison.

\begin{lemma}[$C^0$ estimate]\label{lem:C0_estimate}
The solution $u_S$ satisfies
\begin{equation}\label{eq:C0_ordering}
    \underline{u}_c\le u_S\le Q_c
    \quad\text{in }\Omega_S.
\end{equation}
Consequently, in the far-field region $\rho\ge 2R_c$,
\begin{equation}\label{eq:farfield_C0_squeeze}
    0\le Q_c-u_S\le Q_c-u_\infty^c
    =\Theta_c\rho^{2-n}-\delta_c\rho^{-n}.
\end{equation}
\end{lemma}

\begin{proof}
	The comparison principle gives the ordering.  Since \(Q_c\) is a classical
	solution and the global construction gives \(\underline u_c\le Q_c\), while
	\[
	\underline u_c\le u_S\le Q_c
	\quad\text{on }\partial\Omega_S.
	\]
	The comparison principle gives
	\[
	\underline u_c\le u_S\le Q_c
	\quad\text{in }\Omega_S.
	\]
  Since $\underline{u}_c=u_\infty^c$ for $\rho\ge2R_c$, \eqref{eq:farfield_C0_squeeze} follows from Proposition \ref{prop:smooth_global_subsolution}.
\end{proof}

The estimates near the inner boundary are exactly the local estimates of
Li--Xiao.  The consequence needed below is the following.

\begin{lemma}[Inner boundary estimates]\label{lem:inner_boundary_estimates}
	There exists a constant \(C\), independent of \(S\), such that
	\[
	|Du_S|+|D^2u_S|\le C
	\quad\text{on }\partial D.
	\]
\end{lemma}

\begin{proof}
	The argument is local near \(\partial D\) and is independent of the outer norm
	\(S\).  In a fixed bounded neighborhood of \(\partial D\), the global smooth
	subsolution \(\underline u_c\) agrees with the Li--Xiao near-boundary
	subsolution \(v_0\).  Therefore the boundary gradient estimate and the
	tangential, mixed, and normal second derivative estimates are precisely those
	proved in \cite[Lemmas 3.4, 3.6, and 3.7]{LiXiao2025}.
\end{proof}

Estimates on the outer boundary are also needed.

\begin{lemma}[Outer boundary estimates]\label{lem:outer_C2}
	There exists a constant \(C\), independent of \(S\), such that
	\[
	|Du_S|\le CS,
	\qquad
	|D^2u_S|\le C
	\quad\text{on }\partial E_S.
	\]
\end{lemma}

\begin{proof}
	The estimate is reduced to a fixed outer boundary.  Set
	\[
	U_S(y):=S^{-2}u_S(Sy).
	\]
	Then
	\[
	D_yU_S=S^{-1}D_xu_S,
	\qquad
	D_y^2U_S=D_x^2u_S,
	\]
	and \(U_S\) satisfies the same Hessian equation in
	\(E_1\setminus S^{-1}\overline D\). For \(S\) sufficiently large, the scaled domain \(S^{-1}\overline D\) lies
	strictly inside \(E_1\). Hence a collar neighborhood of the outer boundary \(\partial E_1\), with width
	independent of \(S\),  is contained in
	\(E_1\setminus S^{-1}\overline D\). The boundary estimates near \(\partial E_1\)
	therefore depend only on the fixed local geometry of \(\partial E_1\) and the
	scaled boundary data, not on the inner boundary.
	
	Near \(\partial E_1\), the boundary value is
	\[
	\begin{aligned}
		g_S(y):=S^{-2}u_\infty^c(Sy)
		&=
		\frac12 y^TAy+S^{-1}b\cdot y+S^{-2}c  \\
		&\quad
		-\Theta_cS^{-n}\rho(y)^{2-n}
		+\delta_cS^{-n-2}\rho(y)^{-n}.
	\end{aligned}
	\]
	Since \(S\ge2R_c\) and all parameters in \(u_\infty^c\) are already fixed,
	\[
	\|g_S\|_{C^m(\partial E_1)}\le C_m
	\]
	for every fixed \(m\), with constants independent of \(S\).  The scaled ordering
	\[
	S^{-2}\underline u_c(Sy)\le U_S(y)\le S^{-2}Q_c(Sy)
	\]
	also gives a uniform \(C^0\) bound for \(U_S\) in the fixed collar.
	
	The standard local outer-boundary estimates for the \(k\)-Hessian Dirichlet
	problem on the fixed strictly convex boundary \(\partial E_1\), in the form used
	by Li--Xiao \cite[Lemma 3.4 and Section 3.3.2]{LiXiao2025}, therefore apply.
	This gives
	\[
	|D_yU_S|+|D_y^2U_S|\le C
	\quad\text{on }\partial E_1,
	\]
	with \(C\) independent of \(S\).  Scaling back yields
	\[
	|D_xu_S|\le CS,
	\qquad
	|D_x^2u_S|\le C
	\quad\text{on }\partial E_S.
	\]
This gives the desired estimates.
\end{proof}

The boundary estimates give a global second derivative bound.

\begin{proposition}[Uniform $C^2$ estimate on $\Omega_S$]\label{prop:global_C2_S}
There exists a constant $C$, independent of $S$, such that
\begin{equation}\label{eq:global_C2_S}
    |D^2u_S|\le C
    \quad\text{in }\Omega_S.
\end{equation}
\end{proposition}

\begin{proof}
Let $\widehat F(M)=\sigma_k(\lambda(M))^{1/k}$.  The function $\widehat F$ is concave for \(M\in \mathcal{M}_{k}\).  Differentiating $\widehat F(D^2u_S)=1$ twice and using concavity yields
\begin{equation*}
    \widehat F^{ij}(D^2u_S)(\Delta u_S)_{ij}\ge0.
\end{equation*}
By the maximum principle for the linearized operator,
\begin{equation*}
    \max_{\overline\Omega_S}\Delta u_S
    =\max_{\partial\Omega_S}\Delta u_S.
\end{equation*}
The boundary estimates in Lemmas \ref{lem:inner_boundary_estimates} and \ref{lem:outer_C2} give a uniform bound for the right-hand side.  Hence $\Delta u_S\le C$ in $\Omega_S$.  Since $u_S$ is $k$-admissible, the standard cone inequalities in $\Gamma_k$ imply that a bound for $\sigma_1(\lambda(D^2u_S))=\Delta u_S$ controls all eigenvalues.  Therefore \eqref{eq:global_C2_S} follows.
\end{proof}

It is now possible to pass to the exterior limit.

\begin{proposition}[Large-constant smooth exterior solutions]\label{prop:large_constant_smooth_solution}
Under the hypotheses of Theorem~\ref{thm:smooth_sharp_threshold}, there exists
\(c_0\in\R\) such that for every \(c>c_0\), problem \eqref{eq:problem} admits a
unique \(k\)-admissible solution \(u\in C^\infty(\overline{\Omega})\) satisfying
\eqref{eq:asymptotic}.  
\end{proposition}

\begin{proof}
Let \(c\) be so large that the smooth global subsolution of Proposition
\ref{prop:smooth_global_subsolution} is available.  Let \(S_j\to\infty\) with
\(S_j>2R_c\).  By the Evans--Krylov theorem \cite{Evans1982,Krylov1982} and Schauder estimates, the inner boundary
estimates, and a diagonal Arzel\`a--Ascoli argument, there exist a subsequence,
still denoted \(u_{S_j}\), and a function
\[
    u\in C^\infty(\overline{\Omega})
\]
such that \(u_{S_j}\to u\) in \(C^m(K)\) for every compact set
\(K\Subset\Omega\) and every \(m\ge0\).  Passing to the limit gives
\[
    \sigma_k(\lambda(D^2u))=1,
    \qquad
    \lambda(D^2u)\in\Gamma_k,
    \qquad \text{in }\Omega,
\]
and
\[
    u=\varphi\quad\text{on }\partial D.
\]
Moreover, \eqref{eq:C0_ordering} passes to the limit:
\[
    \underline u_c\le u\le Q_c\quad\text{in }\Omega.
\]
Since \(\underline u_c=u_\infty^c\) for \(\rho\ge2R_c\), we obtain for large \(|x|\),
\[
    0\le Q_c(x)-u(x)
    \le Q_c(x)-u_\infty^c(x)
    =\Theta_c\rho(x)^{2-n}-\delta_c\rho(x)^{-n}.
\]
Because \(\rho\sim |x|\), this gives
\[
    u(x)=Q_c(x)+O(|x|^{2-n}),
    \qquad |x|\to\infty.
\]
Existence follows from the limiting argument above.  The uniqueness follows from the comparison principle.
\end{proof}

The higher-order decay of the large-constant branch is recorded next.

\begin{proof}[Proof of Theorem~\ref{thm:higher_asymptotics}]
  By Proposition
\ref{prop:large_constant_smooth_solution},
\[
    e(x)=u(x)-\left(\frac12x^TAx+b\cdot x+c\right)=O(|x|^{2-n}).
\]
Fix \(x_0\) with \(R_0:=|x_0|\) sufficiently large, and set
\[
    e_{R_0}(y):=\left(\frac4{R_0}\right)^2
    e\left(x_0+\frac{R_0}4y\right),
    \qquad |y|\le2.
\]
Then \(x_0+(R_0/4)B_2\subset\Omega\), and the zero-order decay gives
\[
    \|e_{R_0}\|_{C^0(B_1)}\le CR_0^{-n}.
\]
Moreover, since \(D_y^2e_{R_0}=D_x^2e\) at the corresponding point,
\[
    \widehat F(A+D_y^2e_{R_0})=\widehat F(A)=1.
\]
For \(0\le t\le1\), the eigenvalues of matrices
\[
    A+tD_y^2 e_{R_0}=(1-t)A+tD^2u
\]
stay in a compact subset of \(\Gamma_k\) by convexity, and the global \(C^2\) bound obtained in the annular
approximation.  Thus the rescaled equation is uniformly elliptic and concave on
\(B_1\), with constants independent of \(R_0\).  The interior Evans--Krylov
estimate gives
\[
    \|e_{R_0}\|_{C^{2,\alpha}(B_{3/4})}
    \le C\|e_{R_0}\|_{C^0(B_1)}
    \le CR_0^{-n}.
\]
Differentiating the equation, each first derivative of \(e_{R_0}\)
satisfies a uniformly elliptic linear equation with \(C^\alpha\) coefficients.
Schauder estimates yield, for every \(m\ge1\),
\[
    |D_y^m e_{R_0}(0)|\le C_mR_0^{-n}.
\]
Since
\[
    D_y^m e_{R_0}(0)=4^{2-m}R_0^{m-2}D_x^me(x_0),
\]
we obtain
\[
    |D_x^me(x_0)|\le C_mR_0^{2-n-m}.
\]
This proves
\[
    |x|^{n-2+m}|D^me(x)|\le C_m
\]
for all sufficiently large \(|x|\), and hence \eqref{eq:higher_asymptotics}.
\end{proof}

The section closes with the ordered exterior solvability statement used later in
the interpolation argument. Its proof follows the same annular approximation
mechanism as above, with the prescribed ordered subsolution and supersolution
replacing the particular pair \((\underline u_c,Q_c)\).

\begin{proposition}[Ordered exterior solvability]\label{prop:ordered_solvability}
Let \(D\) and \(\varphi\) satisfy the smooth hypotheses of Theorem~\ref{thm:smooth_sharp_threshold},
and let \(c\in\R\).  Suppose that \(\underline u,\overline u\in C^\infty(\overline{\Omega})\)
satisfy
\[
    \lambda(D^2\underline u)\in\Gamma_k,
    \qquad
    \sigma_k(\lambda(D^2\underline u))\ge1,
\]
\[
    \lambda(D^2\overline u)\in\Gamma_k,
    \qquad
    \sigma_k(\lambda(D^2\overline u))\le1,
\]
in \(\Omega\), and
\begin{equation}\label{eq:ordered_boundaries}
    \underline u\le\overline u\quad\text{in }\Omega,
    \qquad
    \underline u=\varphi\le\overline u\quad\text{on }\partial D.
\end{equation}
Assume also that
\begin{equation}\label{eq:ordered_asymptotic}
    \underline u(x)=Q_c(x)+O(|x|^{2-n}),
    \qquad
    \overline u(x)=Q_c(x)+O(|x|^{2-n}).
\end{equation}
Then \eqref{eq:problem} has a unique \(k\)-admissible solution \(u\in C^\infty(\overline{\Omega})\)
such that
\[
    \underline u\le u\le\overline u\quad\text{in }\Omega
\]
and satisfying \eqref{eq:asymptotic}.
\end{proposition}

\begin{proof}
	Fix \(R\) large and set
	\[
	\Omega_R:=E_R\setminus\overline D .
	\]
	We consider the bounded Dirichlet problem
	\[
	\begin{cases}
		\sigma_k(\lambda(D^2u_R))=1, & \text{in } \Omega_R,\\
		u_R=\varphi, & \text{on } \partial D,\\
		u_R=\underline u, & \text{on } \partial E_R .
	\end{cases}
	\]
	The function \(\underline u\) is a smooth \(k\)-admissible subsolution which agrees
	with the prescribed boundary values. Hence, by Guan's bounded-domain Dirichlet
	theorem with a \(k\)-admissible subsolution \cite{Guan2023}, applied to the smooth
	bounded domain \(\Omega_R\), there exists a smooth \(k\)-admissible solution
	\(u_R\in C^\infty(\overline{\Omega_R})\). Since \(\overline u\) is an
	\(k\)-admissible supersolution and lies above the boundary data, the comparison
	principle gives
	\[
	\underline u\le u_R\le \overline u
	\quad\text{in } \Omega_R .
	\]
	
	Let \(K\subset\overline\Omega\) be fixed. For all sufficiently large \(R\),
	\(K\subset\overline{\Omega_R}\) and \(K\) stays a positive distance away from
	the outer boundary \(\partial E_R\). The ordering above, together with
	\eqref{eq:ordered_asymptotic}, gives a uniform \(C^0\) bound for \(u_R\) on a
	slightly larger compact subset of \(\overline\Omega\), independent of \(R\).
	The local interior estimates, and near \(\partial D\) the fixed
	inner-boundary estimates used in the preceding annular approximation, therefore
	yield, for every \(m\ge0\),
	\[
	\|u_R\|_{C^m(K)}\le C_{m,K},
	\]
	where \(C_{m,K}\) is independent of \(R\). Passing to a diagonal subsequence, we
	obtain a smooth \(k\)-admissible solution \(u\in C^\infty(\overline\Omega)\) of
	\eqref{eq:problem}. Passing the ordering to the limit gives
	\[
	\underline u\le u\le \overline u
	\quad\text{in } \Omega .
	\]
	Since both \(\underline u\) and \(\overline u\) satisfy
	\eqref{eq:ordered_asymptotic}, the same ordering implies that \(u\) satisfies
	\eqref{eq:asymptotic}.
	
	Uniqueness follows by applying the comparison principle on
	\(\Omega_R\) and then letting \(R\to\infty\), since two solutions satisfying
	\eqref{eq:asymptotic} differ by \(o(1)\) on \(\partial E_R\).
\end{proof}

\section{The sharp asymptotic constant}\label{sec:sharp_constant}

The first step is the lower bound for \(k\)-admissible subsolutions.

Let
\begin{equation}\label{eq:c_lower}
	\underline c
	:=\max_{\partial D}(\varphi-Q_0)-\frac{A_\partial^+}{\gamma\mathcal H_0},
\end{equation}
where \(\mathcal H_0\) is defined in  \eqref{eq:H0_intro}, \(A_\partial^+\) in \eqref{eq:A_boundary_plus},  and \(\gamma\)  in \eqref{eq:gamma}.
\begin{proposition}[Nonexistence for very negative constants]\label{prop:lower_bound}
Under the hypotheses of Theorem~\ref{thm:smooth_sharp_threshold}. If
\(c<\underline c\), then there is no \(u\in C^2(\overline\Omega)\) satisfying
\[
    \lambda(D^2u)\in\Gamma_k,
    \qquad
    \sigma_k(\lambda(D^2u))\ge1
    \quad\text{in }\Omega,
\]
and
\[
    u=\varphi\quad\text{on }\partial D,
    \qquad
    u(x)=Q_c(x)+o(1)\quad\text{as }|x|\to\infty.
\]
\end{proposition}

\begin{proof}
Assume that such a subsolution exists and set \(w=u-Q_c\).  Let
\[
    H:=\max_{\partial D}w=\max_{\partial D}(\varphi-Q_0)-c.
\]
If \(H\le0\), then \(c\ge\max_{\partial D}(\varphi-Q_0)\), which is stronger
than \(c\ge\underline c\).  Hence we may assume \(H>0\).  By
Lemma~\ref{lem:subharmonic}, \(L_Aw\ge0\) in \(\Omega\). It follows from \eqref{eq:capacity_problem}  that \(w\le H=Hh\) on \(\partial D\), and at infinity both \(w\) and \(Hh\) tend to zero.  The
comparison principle for the linear operator \(L_A\) gives
\begin{equation}\label{eq:w_le_Hh}
    w\le Hh\quad\text{in }\Omega.
\end{equation}
Choose \(\xi_0\in\partial D\) such that \(w(\xi_0)=H\).  Then
\[
    Hh-w\ge0\quad\text{in }\Omega,
    \qquad
    (Hh-w)(\xi_0)=0.
\]
Recall that \(\nu\) points from \(D\) into \(\Omega\).  Differentiating
\(Hh-w\ge0\) in this \(\nu\)-direction gives
\begin{equation}\label{eq:normal_derivative_bound}
    \partial_\nu w(\xi_0)
    \le H\partial_\nu h(\xi_0)
    \le -\gamma H.
\end{equation}
Since \(w|_{\partial D}\) has a maximum at \(\xi_0\), its intrinsic Hessian on
\(\partial D\) is negative semidefinite at \(\xi_0\).  Hence the Laplace--Beltrami operator on \(\partial D\) satisfies
\begin{equation}\label{eq:boundary_laplacian_negative}
    \Delta_{\partial D}w(\xi_0)\le0,
\end{equation}
Taking the tangential trace in \eqref{eq:ambient_boundary_hessian} at \(\xi_0\)
and using \eqref{eq:H0_intro}, \eqref{eq:normal_derivative_bound}, and
\eqref{eq:boundary_laplacian_negative}, we obtain
\begin{equation}\label{eq:tangential_trace_w_bound}
\begin{aligned}
    \tr_{T_{\xi_0}\partial D}D^2w
    &=\Delta_{\partial D}w(\xi_0)
      +\partial_\nu w(\xi_0)\mathcal H(\xi_0) \\
    &\le -\gamma H\mathcal H_0,
\end{aligned}
\end{equation}
Consequently,
\begin{equation}\label{eq:tangential_trace_u_bound}
\begin{aligned}
    \tr_{T_{\xi_0}\partial D}D^2u
    &=\tr_{T_{\xi_0}\partial D}A
      +\tr_{T_{\xi_0}\partial D}D^2w \\
    &\le A_\partial^+-\gamma H\mathcal H_0.
\end{aligned}
\end{equation}
If \(c<\underline c\), then
\[
    H=\max_{\partial D}(\varphi-Q_0)-c>\frac{A_\partial^+}{\gamma\mathcal H_0}.
\]
Therefore \eqref{eq:tangential_trace_u_bound} gives
\begin{equation}\label{eq:tangential_trace_u_negative}
    \tr_{T_{\xi_0}\partial D}D^2u<0.
\end{equation}
This contradicts \(k\)-admissibility.  Indeed, \(u\in C^2(\overline\Omega)\) and
\(\lambda(D^2u)\in\Gamma_k\) in \(\Omega\), so by continuity
\(\lambda(D^2u(\xi_0))\in\overline\Gamma_k\). Since \(k\ge2\), we have
\(\overline\Gamma_k\subset\overline\Gamma_2\).  Applying
Lemma~\ref{lem:hyperplane_trace} to \(M=D^2u(\xi_0)\) and to the unit normal
\(\nu(\xi_0)\) gives
\[
    \tr_{T_{\xi_0}\partial D}D^2u
    =\tr D^2u(\xi_0)-\nu(\xi_0)^TD^2u(\xi_0)\nu(\xi_0)
    \ge0,
\]
which is incompatible with \eqref{eq:tangential_trace_u_negative}.  Hence no
such $C^2$ subsolution exists when \(c<\underline c\).
\end{proof}

The solvability set can now be defined and its interval property proved.  Define
\begin{equation}\label{eq:solvability_set}
    \mathcal E:=\{c\in\R:\ \text{there exists a \(k\)-admissible }u\in C^\infty(\overline{\Omega})
    \text{ satisfying }\; \eqref{eq:problem}, \eqref{eq:asymptotic}\}.
\end{equation}
By Proposition~\ref{prop:large_constant_smooth_solution}, \(\mathcal E\ne\emptyset\).  By
Proposition~\ref{prop:lower_bound}, \(\mathcal E\) is bounded below.  Set
\begin{equation}\label{eq:Cstar}
    C^*:=\inf\mathcal E.
\end{equation}
The next lemma shows that \(\mathcal E\) is an interval extending to infinity.

\begin{lemma}\label{lem:interpolation}
Suppose \(c_0,c_1\in\mathcal E\) with \(c_0<c_1\).  Then every
\(c\in(c_0,c_1)\) belongs to \(\mathcal E\).
\end{lemma}

\begin{proof}
Let \(u_0,u_1\) be the corresponding \(k\)-admissible solutions.  Fix
\(c\in(c_0,c_1)\) and choose \(\alpha\in(0,1)\) such that
\[
    c=\alpha c_0+(1-\alpha)c_1.
\]
Define
\begin{equation}\label{eq:convex_interpolation}
    \underline u:=\alpha u_0+(1-\alpha)u_1.
\end{equation}
Since \(\Gamma_k\) is convex and \(\sigma_k^{1/k}\) is concave on \(\Gamma_k\),
\[
    \sigma_k(\lambda(D^2\underline u))^{1/k}
    \ge \alpha\sigma_k(\lambda(D^2u_0))^{1/k}
      +(1-\alpha)\sigma_k(\lambda(D^2u_1))^{1/k}
    =1.
\]
Thus \(\underline u\) is a \(k\)-admissible subsolution.  Moreover,
\[
    \underline u=\varphi\quad\text{on }\partial D,
    \qquad
    \underline u(x)=Q_c(x)+O(|x|^{2-n}).
\]
Set
\begin{equation}\label{eq:interpolation_supersolution}
    \overline u:=u_0+(c-c_0).
\end{equation}
Then \(\overline u\) is a \(k\)-admissible solution of the same equation,
\[
    \overline u=\varphi+(c-c_0)\ge\varphi\quad\text{on }\partial D,
    \qquad
    \overline u(x)=Q_c(x)+O(|x|^{2-n}).
\]
The ordering \(\underline u\le\overline u\) is obtained by comparing
\(u_1\) with \(u_0+(c_1-c_0)\).  On \(\partial D\),
\[
    u_1=\varphi\le\varphi+(c_1-c_0)=u_0+(c_1-c_0),
\]
and at infinity the difference tends to zero.  Lemma~\ref{lem:comparison} gives
\[
    u_1\le u_0+(c_1-c_0)\quad\text{in }\Omega.
\]
Therefore
\[
    \underline u
    =\alpha u_0+(1-\alpha)u_1
    \le \alpha u_0+(1-\alpha)(u_0+c_1-c_0)
    =u_0+c-c_0=\overline u.
\]
Proposition~\ref{prop:ordered_solvability} now gives a \(k\)-admissible smooth
solution with asymptotic profile \(Q_c\).  Thus \(c\in\mathcal E\).
\end{proof}

The interval property immediately yields upward closure once one large solvable
constant is available.

\begin{corollary}\label{cor:upward_closed}
If \(c_0\in\mathcal E\), then \((c_0,\infty)\subset\mathcal E\).
\end{corollary}

\begin{proof}
Let \(c>c_0\).  Choose \(c_1>c\) so large that \(c_1\in\mathcal E\), which is
possible by Proposition~\ref{prop:large_constant_smooth_solution}.  Applying
Lemma~\ref{lem:interpolation} to \(c_0\) and \(c_1\) gives \(c\in\mathcal E\).
\end{proof}

It remains to include the endpoint.  The first compactness step gives only an \(o(1)\) asymptotic remainder; the far-field barriers will then improve this to
\(O(|x|^{2-n})\).

\begin{lemma}\label{lem:endpoint}
There exists a \(k\)-admissible solution \(u_*\in C^\infty(\overline{\Omega})\) of
\eqref{eq:problem} satisfying
\[
    u_*(x)=Q_{C^*}(x)+o(1),
    \qquad |x|\to\infty.
\]
\end{lemma}

\begin{proof}
Choose a decreasing sequence \(c_j\in\mathcal E\) with \(c_j\downarrow C^*\),
and let \(u_j\) be the corresponding \(k\)-admissible solution.  If \(i<j\), then
\(c_i>c_j\).  Comparing \(u_j\) with \(u_i\), we have equality on \(\partial D\)
and
\[
    u_j-u_i\to c_j-c_i<0\quad\text{at infinity}.
\]
Hence
\begin{equation}\label{eq:monotone_one}
    u_j\le u_i\quad\text{in }\Omega.
\end{equation}
Next compare \(u_i-(c_i-c_j)\) with \(u_j\).  On \(\partial D\),
\[
    u_i-(c_i-c_j)=\varphi-(c_i-c_j)\le\varphi=u_j,
\]
and at infinity the difference tends to zero.  Therefore
\begin{equation}\label{eq:monotone_two}
    u_i-(c_i-c_j)\le u_j\quad\text{in }\Omega.
\end{equation}
Combining \eqref{eq:monotone_one} and \eqref{eq:monotone_two}, for \(i<j\),
\begin{equation}\label{eq:monotone_squeeze}
    0\le u_i-u_j\le c_i-c_j\quad\text{in }\Omega.
\end{equation}
Thus \(\{u_j\}\) converges uniformly on \(\Omega\) to a continuous function
\(u_*\), and \(u_j\downarrow u_*\).

On each compact subset of \(\Omega\), the estimates are purely local. Indeed,
fix \(K\Subset K'\Subset\Omega\). After fixing the index \(i\), the squeeze
estimate \eqref{eq:monotone_squeeze} gives a uniform \(C^0\) bound for
\(u_j\) on \(K'\), for all large \(j\). By the standard localization and
rescaling argument for the \(k\)-Hessian equation, the interior gradient
estimate and the Pogorelov-type second derivative estimate of Chou--Wang
\cite[Theorems 3.2 and 1.5]{ChouWang2001} apply uniformly, see also the local
estimate argument in Wang--Wang \cite[Lemma 5.4]{WangWang2024}. Hence
\[
\|u_j\|_{C^2(K)}\le C_K,
\]
where \(C_K\) is independent of \(j\).

The estimates are also uniform in a fixed collar of the inner boundary
\(\partial D\). The boundary geometry and the boundary data are fixed: 
\(\partial D\) is smooth and strictly \((k-1)\)-convex, 
\(\varphi\in C^\infty(\partial D)\), and the constants \(c_j\) remain in a
bounded interval. Moreover, the local subsolution \(v_0\) near \(\partial D\)
is fixed and independent of \(j\), while \eqref{eq:monotone_squeeze} gives a
uniform \(C^0\) bound in the collar. Therefore the boundary \(C^2\) estimates of
Caffarelli--Nirenberg--Spruck \cite{CNS1985}, Trudinger
\cite{Trudinger1995}, and Guan \cite{Guan1994}, in the annular form used by
Li--Xiao \cite{LiXiao2025}, give
\[
\|u_j\|_{C^2(K)}\le C_K
\]
for every compact set \(K\subset\overline\Omega\), with \(C_K\) independent of
\(j\).

On each such \(K\), the uniform \(C^2\) bounds make the equation uniformly
elliptic and concave along the sequence. The Evans--Krylov theorem, together with Schauder theory, then yields uniform 
\(C^m\) bounds on compact subsets of \(\overline\Omega\), for every \(m\).
By Arzel\`a--Ascoli and a diagonal argument, a subsequence converges in \(C^m\)
on compact subsets of \(\overline\Omega\), for every \(m\).  Since the uniform
limit is unique by \eqref{eq:monotone_squeeze}, the whole sequence converges
locally smoothly up to \(\partial D\) to \(u_*\).  Consequently
\[
    u_*\in C^\infty(\overline\Omega)
\]
is a \(k\)-admissible solution of \eqref{eq:problem}.

It remains to verify the asymptotic behavior.  Fix \(i\).  Letting
\(j\to\infty\) in \eqref{eq:monotone_one} and \eqref{eq:monotone_two} gives
\[
    u_i-(c_i-C^*)\le u_*\le u_i.
\]
Subtracting \(Q_{C^*}\), we obtain
\begin{equation}\label{eq:endpoint_asymptotic_squeeze}
    u_i-Q_{c_i}
    \le u_*-Q_{C^*}
    \le u_i-Q_{c_i}+c_i-C^*.
\end{equation}
First let \(|x|\to\infty\) in \eqref{eq:endpoint_asymptotic_squeeze}, and then
let \(i\to\infty\).  Since \(u_i-Q_{c_i}\to0\) at infinity and
\(c_i\downarrow C^*\), we get
\[
    u_*(x)=Q_{C^*}(x)+o(1),
    \qquad |x|\to\infty.
\]
This proves the lemma.
\end{proof}

Only the decay order remains to be sharpened.  The proof uses the two-sided
far-field barriers from Lemma~\ref{lem:far_field}.

\begin{lemma}[Improvement from \(o(1)\) to \(O(|x|^{2-n})\)]\label{lem:o_to_O}
Let \(u\in C^\infty(\overline{\Omega})\) be a \(k\)-admissible solution of \eqref{eq:problem}
satisfying
\[
    u(x)=Q_c(x)+o(1),
    \qquad |x|\to\infty.
\]
Then
\[
    u(x)=Q_c(x)+O(|x|^{2-n}),
    \qquad |x|\to\infty.
\]
\end{lemma}

\begin{proof}
	Let \(e=u-Q_c\). Since \(e=o(1)\), choose \(R_0\) large enough so that
	\[
	|e|\le 1 \quad\text{in }\; \R^n\setminus E_{R_0}.
	\]
	Set
	\[
	\Theta:=4R_0^{\,n-2},\qquad \delta:=K R_0^{\,n-2},
	\]
	where \(K>0\) is a fixed structural constant chosen large enough. Then
	\[
	\Theta R_0^{-n}=4R_0^{-2},\qquad
	\delta R_0^{-n-2}=K R_0^{-4},
	\]
	which are small for \(R_0\) large, while
	\[
	\delta R_0^{-n-2}
	=K R_0^{-4}
	\ge C_{\mathrm{ff}}\Theta^2R_0^{-2n}
	=16C_{\mathrm{ff}}R_0^{-4}
	\]
	provided \(K\ge 16C_{\mathrm{ff}}\), where \(C_{\mathrm{ff}}\) is a constant depending only on \(C_1\) and \(C_2\) in \eqref{eq:D2psi_bound}. Hence, by the far-field barrier estimates in
	Lemma~\ref{lem:far_field}, after increasing \(R_0\) if necessary,
	\[
	\underline U=Q_c-\Theta\rho^{2-n}+\delta\rho^{-n},
	\qquad
	\overline U=Q_c+\Theta\rho^{2-n}-\delta\rho^{-n}
	\]
	are respectively a \(k\)-admissible subsolution and a \(k\)-admissible supersolution in
	\(\R^n\setminus E_{R_0}\).
	
	Taking \(R_0\) still larger, we may also assume \(K R_0^{-2}\le 1\). Then, on \(\partial E_{R_0}\),
	\[
	\underline U-Q_c
	=
	-4+KR_0^{-2}
	\le -3,
	\qquad
	\overline U-Q_c
	=
	4-KR_0^{-2}
	\ge 3.
	\]
	Since \(|u-Q_c|\le 1\) on \(\partial E_{R_0}\), it follows that
	\[
	\underline U\le u\le \overline U
	\quad\text{on }\; \partial E_{R_0}.
	\]
	
	Fix \(\varepsilon>0\). Since \(u-Q_c\to0\), \(\underline U-Q_c\to0\), and
	\(\overline U-Q_c\to0\) as \(|x|\to\infty\), the comparison principle applied
	on \(E_{R}\setminus\overline{E}_{R_0}\), with \(\underline U-\varepsilon\) and
	\(\overline U+\varepsilon\), gives
	\[
	\underline U-\varepsilon\le u\le \overline U+\varepsilon
	\quad\text{in }\; E_{R}\setminus\overline{E}_{R_0}
	\]
	for all sufficiently large \(R\). Letting first \(R\to\infty\) and then
	\(\varepsilon\downarrow0\), we obtain
	\[
	\underline U\le u\le \overline U
	\quad\text{on}\; \R^n\setminus E_{R_0}.
	\]
	Therefore
	\[
	|u-Q_c|
	\le \Theta\rho^{2-n}+\delta\rho^{-n}
	\le C\rho^{2-n}
	\quad\text{on}\; \R^n\setminus E_{R_0}.
	\]
	The desired estimate follows.
\end{proof}

The proof of Theorem~\ref{thm:smooth_sharp_threshold} can now be completed.

\begin{proof}[Proof of Theorem~\ref{thm:smooth_sharp_threshold}]
By Proposition~\ref{prop:large_constant_smooth_solution}, \(\mathcal E\ne\emptyset\).  By
Proposition~\ref{prop:lower_bound}, \(\mathcal E\) is bounded below.  Hence
\(C^*=\inf\mathcal E\) is finite.

Let \(c>C^*\).  By the definition of \(C^*=\inf\mathcal E\), we can choose
\(c_0\in\mathcal E\) with
\[
    C^*<c_0<c.
\]
Corollary~\ref{cor:upward_closed} then implies \((c_0,\infty)\subset\mathcal E\),
and therefore \(c\in\mathcal E\).

Lemma~\ref{lem:endpoint} gives a \(k\)-admissible smooth solution \(u_*\) at
\(c=C^*\) with asymptotic remainder \(o(1)\).  Lemma~\ref{lem:o_to_O} improves
this to
\[
    u_*(x)=Q_{C^*}(x)+O(|x|^{2-n}),
    \qquad |x|\to\infty.
\]
Thus \(C^*\in\mathcal E\).  This proves \([C^*,\infty)\subset\mathcal E\).
The reverse inclusion follows from the definition of \(C^*\) as the infimum.
Therefore
\[
    \mathcal E=[C^*,\infty).
\]
This proves existence for exactly the constants \(c\ge C^*\).  If \(c<C^*\),
then \(c\notin\mathcal E\), so no \(k\)-admissible smooth solution with the asymptotic
condition \eqref{eq:asymptotic} exists.  Uniqueness for each \(c\ge C^*\) follows
from Lemma~\ref{lem:comparison}.

The final assertion of the theorem, namely the nonexistence of subsolutions for
all \(c<\underline c\), is exactly Proposition~\ref{prop:lower_bound}.
\end{proof}

\begin{remark}
	The explicit asymptotic estimate \eqref{msexp} follows from Propositions \ref{prop:vis_global_subsolution} and \ref{prop:smooth_global_subsolution}, together with the comparison principle.
\end{remark}

\appendix

\section{A strictly \((n-2)\)-convex domain which is not star-shaped}
\label{app:c-shaped-example}

In this appendix we give a simple example showing that strict
\((n-2)\)-convexity of the boundary does not imply star-shapedness of the
domain.

Let \(\Gamma\subset \mathbb R^2\times\{0\}\subset \mathbb R^n\) be a smooth
embedded arc which is close to a circle of radius \(R\) but has a small gap.
Choose the gap first, and then take \(r/R>0\) sufficiently small so that the
radius-\(r\) tube around \(\Gamma\) is embedded and the two end caps are
disjoint. Near the two endpoints, attach strictly convex caps and smooth the
junctions. The resulting domain \(D\) is a smooth \(C\)-shaped solid tube in
\(\mathbb R^n\). Topologically, \(D\simeq [0,1]\times B^{n-1}\),
and hence \(D\) is contractible.

Away from the two end caps, the boundary of \(D\) is the lateral surface of a
thin tube around the curve \(\Gamma\). Let \(s\) be arclength along \(\Gamma\),
and let \(\kappa_\Gamma(s)\) be the curvature of \(\Gamma\). Choose an
orthonormal normal frame \(E_1(s),\dots,E_{n-1}(s)\) along \(\Gamma\), where
\(E_1(s)\) is the principal normal of the planar curve. The lateral surface may
be parametrized by
\[
X(s,\omega)
=
\gamma(s)+r\sum_{\alpha=1}^{n-1}\omega_\alpha E_\alpha(s),
\qquad
\omega=(\omega_1,\dots,\omega_{n-1})\in S^{n-2}.
\]
With the convention that the unit sphere has positive principal curvatures,
the principal curvatures of the lateral surface are
\[
\kappa_1=\cdots=\kappa_{n-2}=\frac1r,
\]
together with one longitudinal curvature
\[
\kappa_{n-1}
=
-\frac{\kappa_\Gamma(s)\omega_1}
{1-r\kappa_\Gamma(s)\omega_1}.
\]
Here \(\omega_1\) denotes the component of the cross-sectional direction
\(\omega\) along the curvature normal \(E_1(s)\). In dimension \(3\), one has
\(\omega_1=\cos\varphi\), where \(\varphi\) is the angular coordinate on the
circular cross-section.

Let \(K=\|\kappa_\Gamma\|_{L^\infty}\). If \(rK<1\), then
\[
|\kappa_{n-1}|
\le
\frac{K}{1-rK}.
\]
Therefore, for \(1\le j\le n-2\),
\[
\sigma_j(\kappa)
\ge
r^{-j}
\left[
\binom{n-2}{j}
-
\binom{n-2}{j-1}
\frac{rK}{1-rK}
\right].
\]
After choosing \(rK>0\) sufficiently small, the bracket is positive for every
\(1\le j\le n-2\). Thus the lateral surface is strictly \((n-2)\)-convex. The
two end caps are chosen to be strictly convex, and since strict
\((n-2)\)-convexity is an open condition, it is preserved after a sufficiently
small smoothing near the junctions. Consequently, \(\partial D\) is strictly
\((n-2)\)-convex.

We now explain why \(D\) is not star-shaped with respect to any point. Let
\[
\Pi:\mathbb R^n\to\mathbb R^2
\]
be the orthogonal projection onto the \((x_1,x_2)\)-plane. If \(D\) were
star-shaped with respect to some \(p\in D\), then \(\Pi(D)\) would be
star-shaped with respect to \(\Pi(p)\). However, by choosing the opening of the
\(C\)-shaped tube sufficiently small and \(r/R\) sufficiently small, the
projected set \(\Pi(D)\) is a thin \(C\)-shaped planar region with empty
kernel. Indeed, for any \(q\in\Pi(D)\), one can find a point
\(y\in\Pi(D)\) on the opposite side of the \(C\) such that the segment
\([q,y]\) crosses the central empty region and hence is not contained in
\(\Pi(D)\). Thus \(\Pi(D)\) is not star-shaped, and consequently \(D\) is not
star-shaped.

This gives a smooth contractible counterexample:
\[
\partial D \text{ is strictly } (n-2)\text{-convex}
\quad\not\Longrightarrow\quad
D \text{ is star-shaped}.
\]

The following figure illustrates the construction in the
three-dimensional case.

\begin{figure}[H]
	\centering
	\resizebox{0.5\textwidth}{!}{%
		\begin{tikzpicture}[
			x=1cm,y=1cm,
			line cap=round,
			line join=round,
			>=Latex,
			font=\small
			]
			
			\def\th{35}
			\def\Rc{2.05}
			\def\rr{0.50}
			\def\Ro{2.55}
			\def\Ri{1.55}
			
			\path[use as bounding box] (-4.8,-3.0) rectangle (4.8,3.0);
			
			\fill[blue!10]
			({\Ro*cos(\th)},{\Ro*sin(\th)})
			arc[start angle=\th,end angle=360-\th,radius=\Ro]
			arc[start angle=-\th,end angle=180-\th,radius=\rr]
			arc[start angle=360-\th,end angle=\th,radius=\Ri]
			arc[start angle=180+\th,end angle=360+\th,radius=\rr]
			-- cycle;
			
			\draw[blue!75!black, very thick]
			({\Ro*cos(\th)},{\Ro*sin(\th)})
			arc[start angle=\th,end angle=360-\th,radius=\Ro]
			arc[start angle=-\th,end angle=180-\th,radius=\rr]
			arc[start angle=360-\th,end angle=\th,radius=\Ri]
			arc[start angle=180+\th,end angle=360+\th,radius=\rr]
			-- cycle;
			
			\draw[blue!55, dashed, thick]
			({\Rc*cos(\th)},{\Rc*sin(\th)})
			arc[start angle=\th,end angle=360-\th,radius=\Rc];
			
			\node[blue!55] at (-0.10,1.78) {\(\Gamma\)};
			\node[blue!75!black] at (-3.60,-1.75) {\(\Pi(D)\)};
			\node at (-0.15,-0.12) {\(\text{cavity}\)};
			
			\fill[red] (-1.85,0.25) circle (2pt);
			\node[red] at (-2.25,0.55) {\(q=\Pi(p)\)};
			
			\fill[black] (1.45,1.05) circle (2pt);
			\node[above right] at (1.35,1) {\(y\)};
			
			\draw[red, very thick, dashed] (-1.85,0.25) -- (1.45,1.05);
			\node[red] at (-0.05,0.92) {\([q,y]\not\subset \Pi(D)\)};
			
			\draw[red, very thick] (-0.15,-0.28) -- (0.15,0.02);
			\draw[red, very thick] (-0.15,0.02) -- (0.15,-0.28);
			
			\node[align=center] at (3.15,0.05) {\(\text{small}\)\\ \(\text{opening}\)};
			
			\draw[gray!70!black, thick, ->]
			({\Rc*cos(\th)},{\Rc*sin(\th)}) -- ({\Ro*cos(\th)},{\Ro*sin(\th)});
			\node[right] at
			({0.5*(\Rc+\Ro)*cos(\th)-0.05},{0.5*(\Rc+\Ro)*sin(\th)-0.1}) {\(r\)};
			
		\end{tikzpicture}%
	}
	\caption{A three-dimensional illustration of the construction: the boundary
		can be strictly \(1\)-convex, while the domain is not star-shaped.}
	\label{fig:c-shaped-example}
\end{figure}
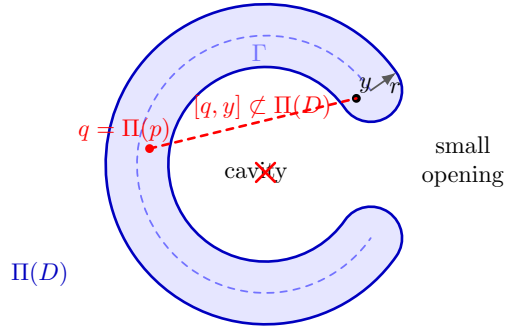

\vskip 0.5cm
\noindent\textbf{Acknowledgements.}
 The authors would like to thank Y. Y. Li and J. G. Xiong for helpful discussions.

\end{document}